\documentclass{amsart}
\usepackage{blkarray}
\usepackage{amsmath,amssymb}
\usepackage{amsthm}
\usepackage{amsfonts}
\usepackage{dsfont}
\usepackage[all]{xy}
\usepackage{supertabular}
\usepackage{longtable}
\usepackage{graphicx}
\DeclareGraphicsExtensions{.pdf,.jpg,.png}
\usepackage{t1enc}
\usepackage[hmarginratio=1:1,top=32mm,columnsep=20pt]{geometry} 

\newtheorem{theorem}{Theorem}[section]
\newtheorem{corollary}[theorem]{Corollary}
\newtheorem{lemma}[theorem]{Lemma}
\newtheorem{proposition}[theorem]{Proposition}

\theoremstyle{definition}
\newtheorem{definition}[theorem]{Definition}
\newtheorem{propo}[theorem]{Property}
\theoremstyle{remark}
\newtheorem{remark}[theorem]{Remark}
\numberwithin{equation}{section}

\newcommand{\K}{\mathbb{K}}

\newcommand{\A}{\mathcal{A}}

\title[]{Structure and classification of Hom-associative algebras.}

\author{Ahmed Zahari}
\address{Universit\'{e} de Haute Alsace,  IRIMAS-d\'epartement de Math\'{e}matiques,
6, rue des Fr\`eres Lumi\`ere F-68093 Mulhouse, France}
\email{zaharymaths@gmail.com}

\author{Abdenacer Makhlouf}
\address{Universit\'{e} de Haute Alsace,  IRIMAS-d\'epartement de Math\'{e}matiques,
6, rue des Fr\`eres Lumi\`ere F-68093 Mulhouse, France}
\email{abdenacer.makhlouf@uha.fr}

\subjclass[]{}

\begin{document}

\begin{abstract}
   The purpose of this paper is to study the structure and the algebraic varieties of Hom-associative algebras. We give  characterize   multiplicative simple Hom-associative algebras and show some examples deforming the $2\times 2$-matrix algebra to simple Hom-associative algebras.  We provide  a classification of $n$-dimensional Hom-associative algebras for  $n\leq3$. Then study their derivations and  compute  small Hom-Type Hochschild cohomology groups.  Furthermore,  we discuss their  irreducible components.
 \end{abstract}

\maketitle
\section*{Introduction}\label{sec:intro} 

The first motivation to study nonassociative Hom-algebras came from quasi-deformations of Lie algebras of vector fields, in particular $q$-deformations of Witt and Virasoro algebras. The deformed algebras arising when replacing usual derivation by a  $\sigma$-derivations are no longer Lie algebras. It was observed in the pioneering works, mainly by physicists,  that in these examples a twisted Jacobi identity holds. Motivated by these examples and their generalization on the one hand, and the desire to be able to treat within the same framework such well-known generalizations of Lie algebras as the color and Lie superalgebras on the other hand, quasi-Lie algebras and subclasses of quasi-hom-Lie algebras and hom-Lie algebras were introduced by Hartwig, Larsson and Silvestrov in \cite{HLS,LS}.
The Hom-associative algebras play the role of associative algebras in the Hom-Lie setting. They were introduced by the second author  and Silvestrov in \cite{MS}. Usual functors between the categories of Lie algebras and associative algebras were extended to Hom-setting, see \cite{Yau-Env2008} for the construction of the enveloping algebra of a Hom-Lie algebra.

A  Hom-associative algebra $(A, \mu, \alpha)$ is consisting of a vector space, a multiplication and a linear self map; It may be viewed as a deformation of an associative algebra, in which the associativity condition is twisted by a linear map $\alpha$  and such that when $\alpha=id$, the Hom-associative algebra degenerates to exactly an  associative algebra. 
We aim in this paper to study the structure of Hom-associative algebras.  We  give a characterization of   multiplicative simple Hom-associative algebras and show some examples deforming the $2\times 2$-matrix algebra to simple Hom-associative algebras.  Moreover we compute some invariants and discuss  irreducible components of the corresponding algebraic varieties.
 Let $A$ be an $n$-dimensional $\mathbb{K}$-linear space and  $\left\{e_1, e_2, \cdots, e_n\right\}$ be a basis of $A$. A Hom-algebra structure on $A$ with product $\mu$ is determined by $n^3$ structure constants 
$\mathcal{C}_{ij}^k$, were $\mu(e_i, e_j)=\sum^n_{k=1}\mathcal{C}_{ij}^ke_k$ and by $\alpha$ which is identified  by ${n^2}$ structure constants
$a_{ij}$, where $\alpha(e_i)=\sum_{j=1}^na_{ji}e_j$. Requiring the algebra structure to be Hom-associative and unital  gives rise to sub-variety $\mathcal{HA}ss_n$ (resp. $\mathcal{UHA}ss_n$) of $k^{n^3+n^2}$. Base changes in $A$ result in the natural transport of structure action of $GL_n(\mathbb{K})$ on $\mathcal{HA}ss_n$. Thus isomorphism classes of $n$-dimensional Hom-algebras are one-to-one correspondence with the orbits of the action of $GL_n(\mathbb{K})$ on $\mathcal{HA}ss_n$. The decomposition of $\mathcal{HA}ss_n$ into  irreducible components with respect to  Zariski topology is called the geometric classification of $n$-dimensional algebras.

The paper is organized as follows. In the first section we give the  basics about Hom-associative algebras and provide some new properties. Moreover, we discuss unital Hom-associative algebras. Section 2 deals with  simple multiplicative Hom-associative algebras. We present one of the main results of this paper, that is a characterization of  simple multiplicative Hom-associative algebras. Indeed, we  show that they are all obtained by twistings of  simple associative algebras. Moreover, we give all simple  Hom-associative algebras, which are related to $2\times 2$ matrix algebra. Section 3 is dedicated to  describe algebraic varieties of Hom-associative algebras and provide classification, up to isomorphism, of 2-dimensional and 3-dimensional Hom-associative algebras. In Section 4, we study  their derivations and twisted derivations, whereas in Section 5, we compute their Hom-type Hochschild cohomology. In the last section, we consider the geometric classification problem, using one-parameter formel deformations, and describe the irreducible components.   
\section{Structure of Hom-associative algebras}\label{sec:prel}
Let $\mathbb{K}$ be an algebraically closed field  of characteristic $0$, $A$ be a linear space over $\mathbb{K}$. We refer to a Hom-algebra by a triple $(A,\mu,\alpha)$, where $\mu:A\times A\rightarrow A$ is a bilinear map (multiplication) and   $\alpha$ is a homomorphism of $A$ (twist map). 

\subsection{Definitions}
\begin{definition}\cite{MS}.
A Hom-associative algebra is a triple $(A, \mu, \alpha)$ consisting of a linear space $A$, a bilinear map $\mu : A\times A\rightarrow A$ and 
a linear space homomorphism $\alpha : A\rightarrow A$ satisfying 
\begin{eqnarray}\label{Homassoc}
 \mu(\alpha(x), \mu(y,z))&=&\mu(\mu(x,y), \alpha(z)).\\
 \alpha(\mu(x,y))&=& \mu(\alpha(x),\alpha(y)).\label{multiplicativity}
\end{eqnarray}
Usually such a Hom-associative algebras are called multiplicative. Since we are dealing only with multiplicative  Hom-associative algebras, we shall call them  Hom-associative algebras for simplicity. 
We denote the set of all Hom-associative algebras  by $\mathcal{HA}ss$. 
In the language of Hopf algebras, the multiplication of a Hom-associative algebra over $A$ consists of a linear map $\mu : A\otimes A\rightarrow A$ and Condition \eqref{Homassoc} writes
$
\mu(\alpha(x)\otimes\mu(y\otimes z))=\mu(\mu(x\otimes y)\otimes\alpha(z)).
$ 
\end{definition}
\begin{definition}\label{d}
 A unital Hom-associative algebra  is given by a quadruple $(A, \mu, \alpha, u)$, where $u\in A$,   such that 
\begin{enumerate}
	\item [$\bullet$] $(A, \mu, \alpha)$ is a Hom-associative algebra,
	\item [$\bullet$] $\mu(x, u)=\mu(u, x)=\alpha(x) \quad \forall x\in A$,
	\item [$\bullet$] $\alpha(u)=u$.
\end{enumerate}
 \end{definition} 
\begin{definition}
Let $(A_1, \mu_1, \alpha_1)$ and $(A_2, \mu_2, \alpha_2)$ be two Hom-associative algebras (resp. unital  Hom-associative algebras with $u_1,u_2$ the units). A linear map $\varphi : A_1\rightarrow A_2$ is called
a  Hom-associative algebras morphism if
\begin{equation}\label{d2} 
\varphi(\mu_1(x,y))=\mu_2(\varphi(x), \varphi(y))\,\text{ and }\,\alpha_2\circ \varphi(x)=\varphi\circ\alpha_1(x),\,\forall x, y\in A.
\end{equation} 
and $\varphi(u_1)=u_2$ for unital algebras.
\end{definition}
In particular, Hom-associative algebras $(A_1,\mu_1,\alpha_1)$ and $(A_2, \mu_2, \alpha_2)$ are isomorphic if $\varphi$ is also bijective.
\subsection{Structure of Hom-associative algebras}
We state in this section some properties on the structure of Hom-associative algebras which are not necessarily multiplicative.
\begin{proposition}[\cite{Yau}]
Let $(A, \mu,\alpha)$ be a Hom-associative algebra and $\beta : A\rightarrow A$ be a Hom-associative algebra morphism. Then $(A,\beta\mu,\beta\alpha)$ is a
Hom-associative algebra. In particular, if $(A,\mu)$ is an associative algebra and $\beta$ is an algebra morphism, then 
$(A, \beta\mu,\beta)$ is a Hom-associative algebra.
\end{proposition}
\begin{definition}\label{Defs1}
Let $(A, \mu, \alpha)$ be a Hom-associative algebra. If there is an associative algebra $(A, \mu')$ such that $\mu(x,y)=\alpha\mu'(x,y),\, \forall x,y\in A$, we say that $(A, \mu, \alpha)$ is of associative type and $(A,\mu')$ is its compatible associative algebra or the untwist of $(A, \mu, \alpha)$.  
\end{definition}
\begin{corollary}
Let $(A, \mu,\alpha)$ be a multiplicative Hom-associative algebra   where $\alpha$ is invertible  then $(A,\mu'= \alpha^{-1}\circ \mu)$ 
is an associative algebra and $\alpha$ is an automorphism with respect to $\mu'$. Hence,  $(A, \mu,\alpha)$ is  of associative type and $(A,\mu'= \alpha^{-1}\circ \mu)$ is its compatible associative algebra. 
\end{corollary}
\begin{proof}
We prove that $(A, \alpha^{-1}\circ \mu)$ is an associative algebra. Indeed,
$$\begin{array}{ll}
\mu'(\mu'(x,y), z)
&=\alpha^{-1}\circ\mu(\alpha^{-1}\mu(x,y),z)
=\alpha^{-1}\circ\mu(\alpha^{-1}\mu(x,y),\alpha^{-1}\circ\alpha(z))\\
&=\alpha^{-2}\circ\mu(\mu(x,y),\alpha(z))
=\alpha^{-2}\circ\mu(\alpha(x),\mu(y,z))
=\alpha^{-1}\circ\mu(x,\alpha^{-1}\circ\mu(x,y))\\
&=\mu'(x,\mu'(y,z)).
\end{array}$$
Moreover, $\alpha$ is an automorphism with respect to $\mu'$. Indeed,
$$\begin{array}{ll}
\mu'(\alpha(x),\alpha(y))
=\alpha^{-1}\circ\mu(\alpha(x),\alpha(y))
=\alpha\circ\alpha^{-1}\circ\mu(x,y)=\alpha\circ\mu'(x,y).
\end{array}$$ 
\end{proof}
\begin{remark}Notice that if $\alpha$ is not invertible, assuming $ \mu=\alpha\tilde{\mu}$ leads to
\begin{eqnarray*}
\mu(\alpha(x),\mu(y,z))
&=&\mu(\mu(x,y), \alpha(z))\\
\alpha\tilde{\mu}(\alpha(x),\alpha\tilde{\mu}(y,z))
&=&\alpha\tilde{\mu}(\alpha\tilde{\mu}(x,y),\alpha(z))\\
\alpha^2(\tilde{\mu}(x,\tilde{\mu}(y,z)))
&=&\alpha^2(\tilde{\mu}(\tilde{\mu}(x,y),z)),\end{eqnarray*}
which means that $\tilde{\mu}$ is associative up to $\alpha^2$.
\end{remark}
\begin{proposition}\label{p3}
Let $(A_1, \mu_1, \alpha_1)$ and $(A_2, \mu_2, \alpha_2)$ be two Hom-associative algebras and  $\phi : A_1\rightarrow A_2$ be an invertible  Hom-associative algebra morphism. If $(A_1, \mu_1, \alpha_1)$ is of associative type and $(A_1, \mu'_1)$  is its compatible associative algebra then $(A_2, \mu_2, \alpha_2)$ is of associative type with compatible associative algebra 
$(A_2, \mu'_2=\phi\circ\mu_1\circ(\phi^{-1}\otimes  \phi^{-1}))$ such that $\phi : (A_1, \mu'_1)\rightarrow (A_2, \mu'_2)$ is an algebra morphism. 
\end{proposition}
\begin{proof}
Because $\phi$ is a  homomorphism from $(A_1, \mu_1, \alpha_1)$ to $(A_2, \mu_2, \alpha_2)$, then 
$\alpha_2\phi=\phi\alpha_1, \, \forall x, y\in A$, $\phi$ defines $\mu_2$ by $\mu_2(\phi(x), \phi(y))=\phi\mu_1(x,y).$ It is easy to 
check that $(A_2, \mu_2)$ is an associative algebra. Furthermore 
$$\begin{array}{ll}
\mu_2(\phi(x),\phi(y))=\phi\circ\mu_1(x, y)
&=\phi\circ\alpha_1\circ\mu'_1(x,y))\\
&=\alpha_2\circ\phi\mu'_1(x,y)=\alpha_2\mu'_2(\phi(x),\phi(y)).
\end{array}$$
We show that $\mu_2$ is an associative algebra  such that $\mu_2(u,v)=\phi\circ\mu_1(\phi^{-1}(u), \phi^{-1}(v))$ with 
$x=\phi^{-1}(u), y=\phi^{-1}(v)$ and $z=\phi^{-1}(w)$ for all $x,y,z \in A$.
$$\begin{array}{ll}
&\mu_2(\mu_2(u,v),w)
=\phi\circ\mu_1(\phi^{-1}\otimes\phi^{-1})(\phi\circ\mu_1(\phi^{-1}\otimes\phi^{-1})(u,v),w)\\
&=\phi\circ\mu_1(\phi^{-1}\otimes\phi^{-1})(\phi\circ\mu_1(\phi^{-1}(u),\phi^{-1}(v)),w)
=\phi\circ\mu_1(\mu_1(\phi^{-1}(u),\phi^{-1}(v)),\phi^{-1}(w))\\
&=\phi\circ\mu_1(\phi^{-1}(u),\mu_1(\phi^{-1}(v),\phi^{-1}(w)))
=\phi\circ\mu_1(\phi^{-1}\otimes\phi^{-1})(\phi\otimes\phi)(\phi^{-1}(u),\mu_1(\phi^{-1}(v),\phi^{-1}(w))\\
&=\phi\circ\mu_1(\phi^{-1}\otimes\phi^{-1})(u,\phi\mu_1(\phi^{-1}(v),\phi^{-1}(w)))
=\mu_2(u,\mu_2(v,w)).
\end{array}$$ 
Hence, $(A_2, \mu_2)$ is  an associative algebra.
\end{proof}

\begin{proposition}\label{p1}
Let $(A, \mu, \alpha)$ be a $n$-dimensional Hom-associative algebra and $\phi : A\rightarrow A$ be an invertible linear map. Then there is an
isomorphism with a n-dimensional Hom-associative algebra  $(A, \mu', \phi\alpha\phi^{-1})$ where\\ 
$\mu'=\phi\circ\mu\circ(\phi^{-1}\otimes\phi^{-1}$). Furthermore, if $\left\{C^k_{ij}\right\}$ are the structure constants of $\mu$ with
respect to the basis $\left\{e_1,\dots,e_n\right\}$,  then $\mu'$ has the same structure constants with respect to the basis 
$\left\{\phi(e_1),\dots,\phi(e_n)\right\}$ when $\phi(e_p)=\sum_{k=1}^na_{kp}e_k$.
\end{proposition}

\begin{proof}
We prove for any invertible linear map $\phi : A\rightarrow A,\,(A, \mu', \phi\alpha\phi^{-1})$ is a Hom-associative algebra.
$$\begin{array}{ll}
&\mu'(\mu'(x, y),\phi\alpha\phi^{-1}(z))
=\phi\mu(\phi^{-1}\otimes\phi^{-1})(\phi\mu(\phi^{-1}\otimes\phi^{-1})(x,y),\phi\alpha\phi^{-1}(z))\\
&=\phi\mu(\mu(\phi^{-1}(x),\phi^{-1}(y)),\alpha\phi^{-1}(z))
=\phi\mu(\alpha\phi^{-1}(x),\mu(\phi^{-1}(y),\phi^{-1}(z)))\\
&=\phi\mu(\phi^{-1}\otimes\phi^{-1})(\phi\otimes\phi)(\alpha\phi^{-1}(x),\mu(\phi^{-1}\otimes\phi^{-1})(y,z)))\\
&=\phi\mu(\phi^{-1}\otimes\phi^{-1})(\phi\alpha\phi^{1}(x),\phi\mu(\phi^{-1}\otimes\phi^{-1})(y,z))
=\mu'(\phi\alpha\phi^{-1}(x),\mu'(y,z)).
\end{array}$$
So $(A, \mu', \phi\alpha\phi^{-1})$ is a Hom-associative algebra.\\
It is also multiplicative. Indeed,  
$$\begin{array}{ll}
&\phi\alpha\phi^{-1}\mu'(x,y)
=\phi\alpha\phi^{-1}\phi\mu(\phi^{-1}\otimes\phi^{-1})(x,y)
=\phi\alpha\mu(\phi^{-1}\otimes\phi^{-1})(x,y)\\
&=\phi\mu(\alpha\phi^{-1}(x),\alpha\phi^{-1}(y))
=\phi\mu(\phi^{-1}\otimes\phi^{-1})(\phi\otimes\phi)(\alpha\phi^{-1}(x),\alpha\phi^{-1}(y))
=\mu'(\phi\alpha\phi^{-1}(x),\phi\alpha\phi^{-1}(y)).
\end{array}$$
Therefore $\phi : (A,\mu,\alpha)\rightarrow(A,\mu',\phi\alpha\phi^{-1})$ is a Hom-associative  algebras morphism, since\\ 
$\phi\circ\mu=\phi\circ\mu\circ(\phi^{-1}\otimes\phi^{-1})\circ(\phi\otimes\phi)=\mu'\circ(\phi\otimes\phi)$ and 
$(\phi\alpha\phi^{-1})\circ\phi=\phi\circ\alpha.$\\
It is easy to see that 
$\left\{\phi(e_i), \cdots, \phi(e_n)\right\}$ is a basis of $A$. For $i,j=1,\cdots,n$, we have 

$\begin{array}{ll}
\mu_2(\phi(e_i),\phi(e_j))
&=\phi\mu_1(\phi^{-1}(e_i), \phi^{-1}(e_j))=\phi\mu(e_i, e_j)=\sum_{k=1}^n\mathcal{C}^k_{ij}\phi(e_k).
\end{array}$ 
\end{proof}

\begin{remark}
A Hom-associative algebra $(A,\mu,\alpha)$ is isomorphic to an associative algebra if and only if $\alpha=id$. Indeed,
$\phi\circ\alpha\phi^{-1}=id$ is equivalent to  $\alpha=id$.
\end{remark}

\begin{remark}
 Proposition \ref{p1} is useful to make a classification of Hom-associative algebras. Indeed, we have to consider the class of morphisms which are conjugate. Representations of these classes are given by Jordan forms of the matrices corresponding to the morphisms. 
 Any $n\times n$ matrix over $\mathbb{K}$ is equivalent, up to basis change,  to  a Jordan canonical form, then we choose $\phi$ such that the matrix of $\phi\alpha\phi^{-1}=\gamma$, where $\gamma$ is a Jordan canonical form. \\
Hence, to obtain the classification, we consider only Jordan forms for the structure map of Hom-associative algebras.
\end{remark}
\begin{proposition}
Let $(A, \mu, \alpha)$ be a Hom-associative algebra. Let $(A, \mu', \phi\alpha\phi^{-1})$ be its isomorphic Hom-associative algebra described
in Proposition \ref{p1}. If $\psi$ is an automorphism of $(A, \mu, \alpha)$, then $\phi\psi\phi^{-1}$ is an automorphism of 
$(A, \mu, \phi\alpha\phi^{-1})$. 
\end{proposition}

\begin{proof}
Note that $\gamma=\phi\alpha\phi^{-1}$. We have\\ $\phi\psi\phi^{-1}\gamma=\phi\psi\phi^{-1}\phi\alpha\phi^{-1}=\phi\psi\alpha\phi^{-1}=\phi\alpha\psi\phi^{-1}=\phi\alpha\phi^{-1}\phi\psi\phi^{-1}=\gamma\phi\psi\phi^{-1}$.

For any $x,y\in A,$
$$\begin{array}{ll}
 \phi\psi\phi^{-1}\mu'(\phi(x),\phi(y))&=\phi\psi\phi^{-1}\phi\mu(x,y)
 =\phi\psi\mu(x,y)=\phi\mu(\psi(x),\psi(y))\\
 &=\mu'(\phi\psi(x), \phi\psi(y))
 =\mu'(\phi\psi\phi^{-1}(\phi(x)), \phi\psi\phi^{-1}(\phi(y))).
 \end{array}$$
 By Definition, $\phi\psi\phi^{-1}$ is an automorphism of $(A, \mu', \phi\alpha\phi^{-1})$.
 \end{proof}
The following characterization was given for Hom-Lie algebras in \cite{Sheng}.
\begin{proposition}
Given two Hom-associative algebras $(A, \mu_A, \alpha)$ and $(B, \mu_B, \beta)$, there is a Hom-associative algebra 
$(A\oplus B, \mu_{A\oplus B}, \alpha+\beta)$, where the bilinear map 
$\mu_{A\oplus B}(., .) : (A\oplus B)\times (A\oplus B)\rightarrow (A\oplus B)$ is given by 
$$\mu_{A\oplus B}(a_1+b_1,a_2+b_2)=(\mu_A(a_1,a_2),\mu_B(b_1,b_2)),\, \forall\, a_1,a_2\in A,\,\forall\, b_1,b_2\in B,$$ and the linear map 
$(\alpha+\beta) : A\oplus B \rightarrow A\oplus B$ is given by $$(\alpha+\beta)(a,b)=(\alpha(a),\beta(b))\,\forall a\in A , b\in B.$$  
\end{proposition}
\begin{proof}
For any $a_i\in A, \, b_i\in B$, by direct computation, we get
$$\begin{array}{ll} 
&\mu_{A\oplus B}((\alpha+\beta)(a_1,b_1), \mu_{A\oplus B}(a_2+b_2,a_3+b_3))=
=\mu_{A\oplus B}((\alpha+\beta)(a_1,b_1), (\mu_A(a_2,a_3), \mu_B(b_2,b_3)))\\
&=\mu_{A\oplus B}((\alpha(a_1),\beta(b_1)), (\mu_A(a_2,a_3), \mu_B(b_2,b_3)))
=(\mu_A(\alpha(a_1),\mu_A(a_2,a_3)),\mu_B(\beta(b_1),\mu_B(b_2,b_3)))\\
&=(\mu_A(\mu_A(a_1,a_2),\alpha(a_3)),\mu_B(\mu_B(b_1,b_2),\beta(b_3)))
=\mu_{A\oplus B}(\mu_{A\oplus B}(a_1+b_1,a_2+b_2),(\alpha+\beta)(a_3,b_3))).
\end{array}$$
This ends the proof.
\end{proof}
\noindent A  Hom-associative algebra morphism  
$\phi : (A, \mu_A, \alpha)\rightarrow (B, \mu_B, \beta)$ is a linear map $\phi : A\rightarrow B$ such that
$\phi\circ \mu_A(a, b)=\mu_B\circ(\phi(a),\phi(b)), \forall a, b\in A, \quad \phi\circ \alpha=\beta\circ \phi.$
Denote by $\xi_\phi\subset A\oplus B$, the graph of linear map $\phi : A\rightarrow B.$
\begin{proposition}
A linear map $\phi : (A, \mu_A, \alpha)\rightarrow (B, \mu_B, \beta)$ is a  Hom-associative algebra morphism if and only if the graph 
$\xi_\phi\subset A\oplus B$ is a Hom-associative subalgebra of $(A\oplus B, \mu_{A\oplus B}, \alpha+\beta)$. 
\end{proposition}  
\begin{proof}
Let $\phi : (A, \mu_A, \alpha)\rightarrow (B, \mu_B, \beta)$ be a  Hom-associative algebra morphism. Then for any $a, b\in A,$ we have 
$$\mu_{A\oplus B}((a, \phi(a)), (b, \phi(b))=(\mu_A(a,b),\mu_B(\phi(a),\phi(b)))=(\mu_A(a,b),\phi\mu_A(a,b)).$$ Thus the graph $\xi_\phi$ is closed under the product $\mu_{A\oplus B}$. Furthermore, since $\phi\circ \alpha=\beta\circ\phi$ we have 
$$(\alpha+\beta)(a, \phi(a))=(\alpha(a),\beta\circ\phi(a))=(\alpha(a), \phi\circ\alpha(a)),$$ which implies that 
$(\alpha+\beta)\subset \xi_\phi.$ Thus $\xi_\phi$ is a Hom-associative subalgebra of $(A\oplus B, \mu_{A\oplus B}, \alpha+\beta)$.

Conversely, if the graph $\xi_\phi\subset A\oplus B$ is a Hom-associative subalgebra of $(A\oplus B, \mu_{A\oplus B}, \alpha+\beta)$, then we have 
 $$\mu_{A\oplus B}((a,\phi(a)),(b,\phi(b)))=(\mu_A(a,b),\mu_B(\phi(a),\phi(b))\in \xi_\phi,$$ which implies that 
 $\mu_B(\phi(a), \phi(b))=\phi\circ\mu_A(a,b).$ Furthermore, $(\alpha+\beta)(\xi_\phi)\subset \xi_\phi$ yields that 
 $$(\alpha+\beta)(a,\phi(a))=(\alpha(a), \beta\circ\phi(a))\in \xi_\phi,$$ which is equivalent to the condition 
$\beta\circ \phi(a)=\phi\circ\alpha(a)$. Therefore, $\phi$ is a  Hom-associative algebra morphism. 
\end{proof}

 \subsection{Unital Hom-associative algebras}
In this section we discuss unital Hom-associative algebras.  We   denote by  $\mathcal{UHA}ss_n$ the set of $n$-dimensional unital Hom-associative algebras.
\begin{proposition}
Let  $(A, \mu,\alpha)$ be a Hom-associative algebra. We set $\tilde{A}=span(A, u)$ the vector space generated by elements of 
$A$ and $u$. Assume  $\mu(x,u)=\mu(u, x)=\alpha(x),\, \forall x\in A$ and $\alpha(u)=u$. Then $(\tilde{A}, \mu,\alpha, u)$ is a unital
Hom-associative algebra.
\end{proposition}
\begin{proof}
It is straightforward  to check the Hom-associativity. For example 
 
$\begin{array}{ll}
\mu(\mu(x,y), \alpha(u))=\mu(\mu(x, y), u)=\alpha(\mu(x,y))=\mu(\alpha(x),\alpha(y))=\mu(\alpha(x),\mu(y,u)).
\end{array}$
\end{proof}
 \begin{remark}
Some unital Hom-associative algebras cannot be obtained as an extension of a non-unital Hom-associative algebra.
\end{remark}
\begin{remark}
Let $(A, \mu, \alpha, u)$ be a $n$-dimensional unital Hom-associative algebra and $\phi : A\rightarrow A$ be an invertible linear map
such that $\phi(u)=u$. Then it is isomorphic to a $n$-dimensional Hom-associative algebra $(A, \mu', \phi\alpha\phi^{-1}, u)$ where 
$\mu'=\phi\circ\mu\circ(\phi^{-1}\otimes\phi^{-1}$). Moreover, if $\left\{C^k_{ij}\right\}$ are the structure constants of $\mu$ with respect to the basis $\left\{e_1,\dots,e_n\right\}$ with $e_1=u$ being the unit, then $\mu'$ has the same structure constants with respect to the basis 
$\left\{\phi(e_1),\dots,\phi(e_n)\right\}$ with $u$ the unit element.

Indeed, we use  Proposition \ref{p1} and Definition \ref{d}. The unit is conserved since
$\mu'(x,e_1)=\phi\circ\mu(\phi^{-1}(x), \phi^{-1}(e_1))=\phi\circ\alpha\circ\phi^{-1}(x).$  
\end{remark}
\begin{proposition}
Let $(A_1, \mu_1, \alpha_1, u_1)$ and $(A_2, \mu_2, \alpha_2,u_2)$ be two unital  Hom-associative algebras. Suppose there exists a Hom-associative algebra morphism $\phi : A_1\rightarrow A_2$  with $\phi(u_1)=u_2$. If $(A_1, \mu'_1, u'_1)$ is an untwist of $(A_1, \mu_1, \alpha_1, u_1)$ then there exists an untwist of $(A_2, \mu_2, \alpha_2, u_2)$ such that $\phi : (A_1, \mu'_1, u'_1)\rightarrow (A_2, \mu'_2, u'_2)$ is an algebra morphism. 
\end{proposition}
\begin{proof}
Since $\phi$ is a homomorphism from $(A_1, \mu_1, \alpha_1, u_1)$ to $(A_2, \mu_2, \alpha_2,u_2)$,  then 
$\alpha_2\phi=\phi\alpha_1,$  and for all $x\in A$ we have $\mu_2(\phi(x),\phi(u_1))=\mu_2(\phi(x),u_2)=\alpha_2\circ\phi(x)$ and 
$\phi\circ\mu_1(x,u_1)=\phi\circ\alpha_1(x)$. By Proposition \ref{p3}, we can see that $(A_2, \mu_2, u_2)$ is also an associative algebra.
Furthermore

$\begin{array}{ll}
\mu'_2(\phi(x), \phi(u_1))=\mu'_2(\phi(x), u_2)
&=\phi\circ\alpha'_1\circ\phi(x)=\phi\circ\alpha_1\circ\mu_1(x,u_1)\\
&=\alpha_2\circ\phi\circ\mu_1(x,u_1)=\alpha_2\circ\mu_2(\phi(x), u_2).
\end{array}$
\end{proof}

\section{Simple Hom-associative algebras}
In this section, we study  and characterize simple multiplicative Hom-associative algebras. Then we provide  exemples by considering $2\times 2$ matrix algebra. This study is inspired by the study of  simple Hom-Lie 
algebras in \cite{Chen}.
\begin{definition}
Let $(A, \mu, \alpha)$ be a Hom-associative algebra. A subspace $H$ of $A$ is called a Hom-associative subalgebra of $(A, \mu, \alpha)$
if $\alpha(H)\subseteq H$ and $\mu(H,H)\subseteq H.$ In particular, a Hom-associative subalgebra $H$ is said to be a two-sided ideal of 
$(A, \mu, \alpha)$ if $\mu(H, A)\subseteq H$ and  $\mu(A, H)\subseteq H$.
\end{definition}
\begin{definition}The set 
\begin{equation}
C(A)=\left\{x\in A|\mu(x,y)=\mu(y,x),\,\mu(\alpha(x),y)=\mu(y,\alpha(x)),\,\forall y\in A\right\}\nonumber 
\end{equation}
is called the center of $(A, \mu, \alpha).$
\end{definition}
Clearly, $C(A)$ is a two-sided ideal.
\begin{lemma}\label{p7}
Let $(A, \mu, \alpha)$ be a multiplicative Hom-associative algebra, then $(Ker(\alpha), \mu, \alpha)$ is a two-sided ideal. 
\end{lemma}
\begin{proof}
Obviously, $\alpha(x)=0\in Ker(\alpha)$ for any $x\in Ker(\alpha).$ Since $\alpha\mu(x, y)=\mu(\alpha(x),\alpha(y)=\mu(0,y)=0$
for any $x\in Ker(\alpha)$ and $y\in A$, we get $\mu(x, y)\in Ker(\alpha).$ \\
On the other hand, we have  $\alpha(y)=0\in Ker(\alpha)$ for any $y\in Ker(\alpha).$ Since  
$\alpha\mu(x, y)=\mu(\alpha(x),\alpha(y))=\mu(x,0)=0$ for any $x\in Ker(\alpha)$ and $y\in A$, we get $\mu(x, y)\in Ker(\alpha).$ 
Therefore, $(Ker(\alpha), \mu, \alpha)$ is a two-sided ideal of $(A, \mu, \alpha)$.
\end{proof} 
\begin{definition}
Let $(A, \mu, \alpha)\, (\alpha\neq 0)$ be a non trivial Hom-associative algebra. It is said to be a simple Hom-associative algebra if
it has no proper two-sided ideal. 
\end{definition}

\begin{theorem}
Let $(A, \mu, \alpha)$ be a finite dimensional simple Hom-associative algebra. Then $\alpha$ is an automorphism,  the Hom-associative algebra is of associative type with a simple compatible associative algebra.
\end{theorem}
\begin{proof}
According to Lemma \ref{p7}, $Ker(\alpha)$ is a two-sided ideal. Since the Hom-associative algebra is simple, either 
$Ker(\alpha)=\left\{0\right\}$ or $Ker(\alpha)=A$. The Hom-associative algebra is nontrivial, therefore $Ker(\alpha)\neq A$.

Thus, $A$ is of associative type. 
Let $(A, \mu'=\alpha^{-1}\mu )$ be the  induced associative algebra of the multiplicative simple  Hom-associative algebra ($A, \mu,\alpha$).
Clearly, $\alpha$ is both an automorphism of ($A, \mu,\alpha$) and ($A, \mu'$).  Indeed  $\alpha\mu'(x,y)=\alpha\alpha^{-1}\mu(x,y)=\alpha^{-1}\mu(\alpha(x),\alpha(y))=\mu'(\alpha(x),\alpha(y))$. 

Suppose that $A_1\neq 0$, is the maximal two-sided ideal of ($A, \mu'$). Because $\alpha(A_1)$ is also a two-sided ideal of 
($A, \mu'$), then $\alpha(A_1)\subseteq A_1$. Moreover, $$\mu(A_1,A)=\alpha\mu'(A_1,A)\subseteq \alpha(A_1)\subseteq A_1$$ and 
$$\mu(A, A_1)=\alpha\mu'(A, A_1)\subseteq \alpha(A_1)\subseteq A_1.$$ So $A_1$ is a two-sided ideal of $(A, \mu, \alpha)$. Then $A_1=A$,
and we have $$\mu(A,A)=\mu(A_1,A)=\alpha\mu'(A_1,A)\varsubsetneqq\alpha(A_1)\subseteq A_1=A$$ and 
$$\mu(A,A)=\mu(A, A_1)=\alpha\mu'(A, A_1)\varsubsetneqq\alpha(A_1)\subseteq A_1=A.$$ Furthermore, since $(A, \mu, \alpha)$ is a
multiplicative simple Hom-associative algebra, we clearly have $\mu(A,A)=A$. It is contradiction. Hence $A_1=0$.     
\end{proof} 

By the above theorem, there exists an induced associative algebra for any multiplicative simple Hom-associative algebra $(A, \mu, \alpha)$
and $\alpha$ is an automorphism of the induced associative algebra, in addition  to this their products are mutually determined.
\begin{theorem}
Two simple Hom-associative algebras $(A_1, \mu_1, \alpha)$ and $(A_2, \mu_2, \beta)$ are isomorphic if and only if there exists an associative algebra isomorphism $\varphi$ : $A_1\rightarrow A_2$ (between their induced associative algebras) satisfying 
$\varphi\circ\alpha=\beta\circ\varphi.$ In other words, the two associative algebra automorphisms $\alpha$ and $\beta$ are conjugate.  
\end{theorem} 
\begin{proof}
Let $(A_1, \tilde{\mu}_1)$ and $(A_2, \tilde{\mu}_2)$ be the  induced associative algebras of $(A_1, \mu_1, \alpha)$ and $(A_2, \mu_2, \beta)$, respectively.
Suppose $\varphi : (A_1, \mu_1, \alpha)\rightarrow (A_2, \mu_2, \beta)$ is an isomorphism of Hom-associative algebras, then 
$\varphi\circ\alpha=\beta\circ\varphi$, thus $\varphi\circ\alpha^{-1}=\beta^{-1}\circ\varphi$. Moreover, 
$$
\varphi\tilde{\mu}_1(x, y)=\varphi\circ\alpha^{-1}\circ\alpha\tilde{\mu}_1(x,y)
=\varphi\circ\alpha^{-1}\mu_1(x,y)=\beta^{-1}\circ\varphi\mu_1(x,y)
=\beta^{-1}(\mu_2(\varphi(x),\varphi(y)))=\tilde{\mu}_2(\varphi(x),\varphi(y)).
$$   
So, $\varphi$ is an isomorphism between the two induced associative algebras. 

On the other hand, if there exists an isomorphism $\varphi$ between the induced associative algebras 
$(A_1,\tilde{\mu}_1)$ and $(A_2,\tilde{\mu}_2)$ 
such that $\varphi\circ\alpha=\beta\circ\varphi$, then\\ 
 $  \varphi\mu_1(x, y)=\varphi\circ\alpha\tilde{\mu}_1(x,y)
=\beta\circ\tilde{\mu}_2(\varphi(x), \varphi(y))
=\beta(\mu_2(\varphi(x),\varphi(y))=\mu_2(\varphi(x), \varphi(y)).
$   
\end{proof}
\subsection{Examples of simple Hom-associative algebras}

 We consider the simple associative algebra defined by
$2\times 2$ matrices, which we denote by $\mathcal{M}_2$. 
Let $\mathcal{B}=\left\{E_{ij}\right\}_{i=1,2 \atop j=1,2}$ be the canonical basis given by elementary matrices.
We seek first  for algebra morphisms $\varphi$ of $\mathcal{M}_2$, that is linear maps such that   
 $$\varphi(E_{ij}).\varphi(E_{kl})=\varphi(E_{ij}.E_{kl})=\delta_{jk}\varphi(E_{il}).$$
 Then we apply the previous theorem to construct families of 4-dimensional simple Hom-associative algebras. We obtain by straightforward calculation the 
following algebra morphisms where  $\delta_{ij}$ is the Kronecker symbol.

\begin{small}
\noindent\textbf{Morphism 1}    
$$\left\{\begin{array}{c} 
\begin{array}{ll} 
\varphi(E_{11})=E_{11}-i\frac{\sqrt{\beta_2}}{\sqrt{\beta_1}}E_{21}
&\varphi(E_{12})=i\sqrt{\beta_1}\sqrt{\beta_2}E_{11}+\beta_1E_{12}+\beta_2E_{21}-i\sqrt{\beta_1}\sqrt{\beta_2}E_{22}\\
\varphi(E_{21})=\frac{E_{21}}{\beta_1}
&\varphi(E_{22})=i\frac{\sqrt{\beta_2}}{\sqrt{\beta_1}}E_{21}+E_{22}
\end{array} 
\end{array}\right.$$
\textbf{Morphism 2}    
$$\left\{\begin{array}{c} 
\begin{array}{ll}  
\varphi(E_{11})=E_{11}+i\frac{\sqrt{\beta_2}}{\sqrt{\beta_1}}E_{21}
&\varphi(E_{12})=-i\sqrt{\beta_1}\sqrt{\beta_2}E_{11}+\beta_1E_{12}+\beta_2E_{21}+i\sqrt{\beta_1}\sqrt{\beta_2}E_{22}\\
\varphi(E_{21})=\frac{E_{21}}{\beta_1}
&\varphi(E_{22})=-i\frac{\sqrt{\beta_2}}{\sqrt{\beta_1}}E_{21}+E_{22}
\end{array}
\end{array}\right.$$
\textbf{Morphism 3}  
$$\left\{\begin{array}{c} 
\begin{array}{ll} 
\varphi(E_{11})=E_{11}-\lambda_1E_{21}
&\varphi(E_{12})=-\frac{\beta_2}{\lambda_1}E_{11}-\frac{\beta_2}{\gamma^2_2}E_{12}+\beta_2E_{21}+\frac{\beta_2}{\lambda_1}E_{22}\\
\varphi(E_{21})=-\frac{\lambda^2_1}{\beta_2}E_{21}
&\varphi(E_{22})=\lambda_1E_{21}+E_{22}
\end{array}   
\end{array}\right.$$
\textbf{Morphism 4}  
$$\left\{\begin{array}{c}
\begin{array}{ll}    
\varphi(E_{11})=E_{11}-\lambda_1E_{21}
&\varphi(E_{12})=\beta_1\lambda_1E_{11}+\beta_1E_{12}-\beta_1\lambda_1^2E_{21}-\beta_1E_{22}\\
\varphi(E_{21})=\frac{E_{21}}{\beta_1}
&\varphi(E_{22})=\lambda_1E_{21}+E_{22}
\end{array}
\end{array}\right.,$$
\textbf{Morphism 5}  
$$\left\{\begin{array}{c}
\begin{array}{ll}    
\varphi(E_{11})=E_{11}+\beta_3\gamma_1E_{21}
&\varphi(E_{12})=-\beta_3E_{11}+\frac{E_{12}}{\gamma_1}-\beta_3\gamma_1E_{21}+\beta_3E_{22}\\
\varphi(E_{21})=\gamma_1E_{21}
&\varphi(E_{22})=-\beta_3\gamma_1E_{21}+E_{22}
\end{array}
\end{array}\right.,$$
\textbf{Morphism 6}  
$$\left\{\begin{array}{c}
\begin{array}{ll}    
\varphi(E_{11})=i\sqrt{\beta_2}\sqrt{\gamma_1}E_{21}+E_{22}
&\varphi(E_{12})=\beta_2E_{21}\\
\varphi(E_{21})=i\frac{\sqrt{\gamma_1}}{\sqrt{\beta_2}}E_{11}+\frac{E_{12}}{\beta_2}+\gamma_1E_{21}-i\frac{\sqrt{\gamma_1}}{\sqrt{\beta_2}}E_{22}
&\varphi(E_{22})=E_{11}-i\sqrt{\beta_2}\sqrt{\gamma_1}E_{21}
\end{array}
\end{array}\right.,$$
\textbf{Morphism 7}  
$$\left\{\begin{array}{c}
\begin{array}{ll}    
\varphi(E_{11})=\frac{\beta_4}{\beta_2}E_{12}+E_{22}
&\varphi(E_{12})=\beta_4E_{11}-\frac{\beta^2_4}{\beta_2}E_{12}+\beta_2E_{21}-\beta_4E_{22}\\
\varphi(E_{21})=\frac{E_{12}}{\beta_2}
&\varphi(E_{22})=E_{11}-\frac{\beta_4}{\beta_2}E_{12}
\end{array}
\end{array}\right.,$$
\textbf{Morphism 8}  
$$\left\{\begin{array}{c}
\begin{array}{ll}    
\varphi(E_{11})=E_{11}+\gamma_2E_{12}
&\varphi(E_{12})=\frac{E_{12}}{\gamma_1}\\
\varphi(E_{21})=-\gamma_2E_{11}-\frac{\gamma^2_2}{\gamma_1}E_{12}+\gamma_1E_{21}+\gamma_2E_{22}
&\varphi(E_{22})=-\frac{\beta_2}{\beta_1}E_{12}+E_{22}
\end{array}
\end{array}\right.,$$
\textbf{Morphism 9}  
$$\left\{\begin{array}{c}
\begin{array}{ll}    
\varphi(E_{11})=-\gamma_2E_{21}+E_{22}
&\varphi(E_{12})=\frac{E_{21}}{\gamma_4}\\
\varphi(E_{21})=-\gamma_2E_{11}+\gamma_4E_{12}-\frac{\gamma^2_2}{\gamma_4}E_{21}+\gamma_2E_{22}
&\varphi(E_{22})=E_{11}+\frac{\gamma_2}{\gamma_4}E_{21}
\end{array}
\end{array}\right.,$$
\end{small}
where $\beta_1,\beta_2,\beta_3,\beta_4,\lambda_1,\lambda_2,\lambda_3,\lambda_4,\gamma_1,\gamma_2,\gamma_3,\gamma_4\in \mathbb{C}$ are parameters.

They lead to simple Hom-associative algebras $(\mathcal{M}_2, \ast, \varphi)$ where $E_{ij}\ast E_{pq}=\varphi(E_{ij}E_{pq}).$\\
Therefore, the multiplication tables are given as follows : 

\begin{small}
\noindent \textbf{Algebra 1}    
$$\left\{\begin{array}{c}  
\begin{array}{ll}
E_{11}\ast E_{11}=E_{11}-i\frac{\sqrt{\beta_2}}{\sqrt{\beta_1}}E_{21}
&E_{11}\ast E_{12}=i\sqrt{\beta_1}\sqrt{\beta_2}E_{11}+\beta_1E_{12}+\beta_2E_{21}-i\sqrt{\beta_1}\sqrt{\beta_2}E_{22}\\
E_{12}\ast E_{21}=E_{11}-i\frac{\sqrt{\beta_2}}{\sqrt{\beta_1}}E_{21}
&E_{12}\ast E_{22}=i\sqrt{\beta_1}\sqrt{\beta_2}E_{11}+\beta_1E_{12}+\beta_2E_{21}-i\sqrt{\beta_1}\sqrt{\beta_2}E_{22}\\
E_{21}\ast E_{11}=\frac{E_{21}}{\beta_1}
&E_{21}\ast E_{12}=-i\frac{\sqrt{\beta_2}}{\sqrt{\beta_1}}E_{21}+E_{22}\\
E_{22}\ast E_{21}=\frac{E_{11}}{\beta_1}
&E_{22}\ast E_{22}=-i\frac{\sqrt{\beta_2}}{\sqrt{\beta_1}}E_{21}+E_{22}
\end{array}
\end{array}\right.$$
\textbf{Algebra 2}    
$$\left\{\begin{array}{c}
\begin{array}{ll}  
E_{11}\ast E_{11}=E_{11}+i\frac{\sqrt{\beta_2}}{\sqrt{\beta_1}}E_{21}
&E_{11}\ast E_{12}=-i\sqrt{\beta_1}\sqrt{\beta_1}E_{11}+\beta_1E_{12}+\beta_2E_{21}+i\sqrt{\beta_1}\sqrt{\beta_1}E_{22}\\
E_{12}\ast E_{21}=E_{11}+i\frac{\sqrt{\beta_1}}{\sqrt{\beta_1}}E_{21}
&E_{12}\ast E_{22}=-i\sqrt{\beta_1}\sqrt{\beta_2}E_{11}+\beta_2E_{12}+\beta_2E_{21}+i\sqrt{\beta_1}\sqrt{\beta_2}E_{22}\\
E_{21}\ast E_{11}=\frac{E_{21}}{\beta_1}\
&E_{21}\ast E_{12}=-i\frac{\sqrt{\beta_2}}{\sqrt{\beta_1}}E_{21}+E_{22}\\
E_{22}\ast E_{21}=\frac{E_{21}}{\beta_1}
&E_{22}\ast E_{22}=-i\frac{\sqrt{\beta_2}}{\sqrt{\beta_1}}E_{21}+E_{22}
\end{array}
\end{array}\right.$$
\textbf{Algebra 3}    
$$\left\{\begin{array}{c} 
\begin{array}{ll}   
E_{11}\ast E_{11}=E_{11}-\lambda_1E_{21}
&E_{11}\ast E_{12}=-\frac{\beta_2}{\lambda_1}E_{11}-\frac{\beta_2}{{\gamma_2}^2}E_{12}+\beta_2E_{21}+\frac{\beta_2}{\lambda_1}E_{22}\\
E_{12}\ast E_{21}=E_{11}-\lambda_1E_{21}
&E_{12}\ast E_{22}=-\frac{\beta_2}{\lambda_1}E_{11}-\frac{\beta_2}{{\gamma_2}^2}E_{12}+\beta_2E_{21}+\frac{\beta_2}{\lambda_1}E_{22}\\
E_{21}\ast E_{11}=-\frac{\lambda_1^2}{\beta_2}E_{21}
&E_{21}\ast E_{12}=\lambda_1E_{21}+E_{22}\\
E_{22}\ast E_{21}=-\frac{\lambda_1^2}{\beta_2}E_{21}
&E_{22}\ast E_{22}=\lambda_1E_{21}+E_{22}.
\end{array}  
\end{array}\right.$$
\textbf{Algebra 4}  
$$\left\{\begin{array}{c} 
\begin{array}{ll} 
E_{11}\ast E_{11}=E_{11}-\lambda_1E_{21}
&E_{11}\ast E_{12}=\beta_1\lambda_1E_{11}+\beta_1E_{12}-\beta_1\lambda_1^2E_{21}-\beta_1E_{22}\\
E_{12}\ast E_{21}=E_{11}-\lambda_1E_{21}
&E_{12}\ast E_{22}=\beta_1\lambda_1E_{11}+\beta_1E_{12}-\beta_1\lambda_1^2E_{21}-\beta_1E_{22}\\
E_{21}\ast E_{11}=\frac{E_{21}}{\beta_1}
&E_{21}\ast E_{12}=\lambda_1E_{21}+E_{22}\\
E_{22}\ast E_{21}=\frac{E_{21}}{\beta_1}
&E_{22}\ast E_{22}=\lambda_1E_{21}+E_{22}
\end{array}
\end{array}\right.$$
\textbf{Algebra 5}  
$$\left\{\begin{array}{c} 
\begin{array}{ll} 
E_{11}\ast E_{11}=E_{11}+\beta_3\gamma_1E_{21}
&E_{11}\ast E_{12}=-\beta_3E_{11}+\frac{E_{12}}{\gamma_1}-\beta_3\gamma_1E_{21}+\beta_3E_{22}\\
E_{12}\ast E_{21}=E_{11}+\beta_3\gamma_1E_{21}
&E_{12}\ast E_{22}=-\beta_3E_{11}+\frac{E_{12}}{\gamma_1}-\beta_3\gamma_1E_{21}+\beta_3E_{22}\\
E_{21}\ast E_{11}=\gamma_1E_{21}
&E_{21}\ast E_{12}=-\beta_3\gamma_1E_{21}+E_{22}\\
E_{22}\ast E_{21}=\gamma_1E_{21}
&E_{22}\ast E_{22}=-\beta_3\gamma_1E_{21}+E_{22}
\end{array}
\end{array}\right.$$
\textbf{Algebra 6}  
$$\left\{\begin{array}{c} 
\begin{array}{ll} 
E_{11}\ast E_{11}=i\sqrt{\beta_2}\sqrt{\gamma_1}E_{21}+E_{22}
&E_{11}\ast E_{12}=\beta_2E_{21}\\
E_{12}\ast E_{21}=i\sqrt{\beta_2}\sqrt{\gamma_1}E_{21}+E_{22}
&E_{12}\ast E_{22}=\beta_2E_{21}\\
E_{21}\ast E_{11}=i\frac{\sqrt{\gamma_1}}{\sqrt{\beta_2}}E_{11}+\frac{E_{12}}{\beta_2}+\gamma_1E_{21}-i\frac{\sqrt{\gamma_1}}{\sqrt{\beta_2}}E_{22}
&E_{21}\ast E_{12}=E_{11}-i\sqrt{\beta_2}\sqrt{\gamma_1}E_{21}\\
E_{22}\ast E_{21}=i\frac{\sqrt{\gamma_1}}{\sqrt{\beta_2}}E_{11}+\frac{E_{12}}{\beta_2}+\gamma_1E_{21}-i\frac{\sqrt{\gamma_1}}{\sqrt{\beta_2}}E_{22}
&E_{22}\ast E_{22}=E_{11}-i\sqrt{\beta_2}\sqrt{\gamma_1}E_{21}
\end{array}
\end{array}\right.$$
\textbf{Algebra 7}  
$$\left\{\begin{array}{c} 
\begin{array}{ll} 
E_{11}\ast E_{11}=\frac{\beta_4}{\beta_2}E_{12}+E_{22}
&E_{11}\ast E_{12}=\beta_4E_{11}-\frac{\beta^2_4}{\beta_2}E_{12}+\beta_2E_{21}-\beta_4E_{22}\\
E_{12}\ast E_{21}=\frac{\beta_4}{\beta_2}E_{12}+E_{22}
&E_{12}\ast E_{22}=\beta_4E_{11}-\frac{\beta^2_4}{\beta_2}E_{12}+\beta_2E_{21}-\beta_4E_{22}\\
E_{21}\ast E_{11}=\frac{E_{12}}{\beta_2}
&E_{21}\ast E_{12}=E_{11}-\frac{\beta_4}{\beta_2}E_{12}\\
E_{22}\ast E_{21}=\frac{E_{12}}{\beta_2}
&E_{22}\ast E_{22}=E_{11}-\frac{\beta_4}{\beta_2}E_{12}
\end{array}
\end{array}\right.$$
\textbf{Algebra 8}  
$$\left\{\begin{array}{c} 
\begin{array}{ll} 
E_{11}\ast E_{11}=E_{11}+\gamma_2E_{12}
&E_{11}\ast E_{12}=\frac{E_{12}}{\gamma_1}\\
E_{12}\ast E_{21}=E_{11}+\gamma_2E_{12}
&E_{12}\ast E_{22}=\frac{E_{12}}{\gamma_1}\\
E_{21}\ast E_{11}=-\gamma_2E_{11}-\frac{\gamma^2_2}{\gamma_1}E_{12}+\gamma_1E_{21}+\gamma_2E_{22}
&E_{21}\ast E_{12}=-\frac{\beta_2}{\beta_1}E_{12}+E_{22}\\
E_{22}\ast E_{21}=-\gamma_2E_{11}-\frac{\gamma^2_2}{\gamma_1}E_{12}+\gamma_1E_{21}+\gamma_2E_{22}
&E_{22}\ast E_{22}=-\frac{\beta_2}{\beta_1}E_{12}+E_{22}
\end{array}
\end{array}\right.$$

\textbf{Algebra 9}  
$$\left\{\begin{array}{c} 
\begin{array}{ll} 
E_{11}\ast E_{11}=-\gamma_2E_{21}+E_{22}
&E_{11}\ast E_{12}=\frac{E_{21}}{\gamma_4}\\
E_{12}\ast E_{21}=-\gamma_2E_{21}+E_{22}
&E_{12}\ast E_{22}=\frac{E_{21}}{\gamma_4}\\
E_{21}\ast E_{11}=-\gamma_2E_{11}+\gamma_4E_{12}-\frac{\gamma^2_2}{\gamma_4}E_{21}+\gamma_2E_{22}
&E_{21}\ast E_{12}=E_{11}+\frac{\gamma_2}{\gamma_4}E_{21}\\
E_{22}\ast E_{21}=-\gamma_2E_{11}+\gamma_4E_{12}-\frac{\gamma^2_2}{\gamma_4}E_{21}+\gamma_2E_{22}
&E_{22}\ast E_{22}=E_{11}+\frac{\gamma_2}{\gamma_4}E_{21}
\end{array}
\end{array}\right.$$
\end{small}
\section{Algebraic varieties of Hom-associative algebras and Classification} 
In this section, we deal with algebraic varieties of Hom-associative algebras with a fixed dimension. A Hom-associative algebra is identified with its structure constants with respect to a fixed basis. Their set corresponds to an algebraic variety where the ideal is generated by polynomials corresponding to the Hom-associativity condition.
\subsection{Algebraic varieties $\mathcal{HA}ss_n$}
Let $A$ be a $n$-dimensional $\K$-linear space and $\left\{e_1, \cdots,e_n\right\}$ be a basis of $A$. A Hom-algebra structure on $A$ with product 
$\mu$ determined by $n^3$ structure constants $\mathcal{C}_{ij}^k$ where  
$\mu(e_i, e_j)=\sum_{k=1}^n\mathcal{C}_{ij}^ke_k$ and a structure map $\alpha$ determined  by $n^2$ structure constants $a_{ji}$, where 
$\alpha(e_i)=\sum_{j=1}^na_{ji}e_j$.\\
%
If we require this algebra structure to be Hom-associative, then this limits the set 
of structure constants ($\mathcal{C}_{ij}^k, a_{ij})$ to a cubic sub-variety of the affine algebraic variety $\mathbb{K}^{n^3+n^2}$ defined by the following polynomial equations system : 
\begin{equation}\label{s1}
\left\{\begin{array}{c}  
\sum_{l=1}^n\sum_{m=1}^na_{il}\mathcal{C}_{jk}^m\mathcal{C}^s_{lm}-a_{mk}\mathcal{C}_{ij}^l\mathcal{C}_{lm}^s=0,\quad i,j,k,s=1,\cdots, n.\\
\sum_{p=1}^na_{sp}\mathcal{C}_{ij}^p-\sum_{p=1}^n\sum_{q=1}^na_{pi}a_{qj}\mathcal{C}_{pq}^s=0,\quad i,j, s=1,\cdots,n.
\end{array}\right.
\end{equation}
Moreover if $\mu$ is commutative, 
we have $\mathcal{C}_{ij}^k=\mathcal{C}_{ji}^k\quad i, j, k=1,\cdots,n.$ 

The first set of equation correspond to the Hom-associative condition 
$\mu(\alpha(e_i),\mu(e_j,e_k))=\mu(\mu(e_i,e_j),\alpha(e_k))$ and the second set to multiplicativity condition 
$\alpha\circ\mu(e_i,e_j)=\mu(\alpha(e_i), \alpha(e_j)).$
We denote by $\mathcal{HA}ss_n$ the set of all $n$-dimensional multiplicative Hom-associative algebras.  

\noindent Assume that $e_1=u$, the unit,  in the basis $\mathcal{B}$. It turns out that in addition to the system $(\ref{s1})$, we have the following condition with respect to  unitality : 
$u_1.e_i=e_i.u_1=\alpha(e_i)\Rightarrow \sum_{k=1}^n\mathcal{C}_{1i}^ke_k=\sum_{k=1}^n\mathcal{C}_{i1}^ke_k=\sum_{k=1}^na_{ki}e_k$,  
that is 
\begin{equation}
\mathcal{C}_{i1}^k=\mathcal{C}_{1i}^k=a_{ki}\quad \forall i, k. 
\end{equation} 
We denote by  $\mathcal{UHA}ss_n$ the  algebraic varieties of $n$-dimensional unital Hom-associative algebras.
\subsection{Action of linear group on the algebraic varieties  $\mathcal{HA}ss_n$ }
The group $GL_n(\mathbb{K})$ acts on the algebraic varieties of  Hom-structures by the so-called transport of structure action defined as follows. Let $A=(A,\mu,\alpha )$ be a $n$-dimensional  Hom-associative algebra defined by multiplication $\mu$ and a linear map $\alpha$. Given $f\in GL_n(\mathbb{K})$, the action $f\cdot A$ transports the structure, 
$$\begin{array}{lll}
\Theta : & GL_n(\mathbb{K})\times \mathcal{HA}ss_n&\longrightarrow\mathcal{HA}ss_n \\
                           \ &      (f, (A,\mu, \alpha))&\longmapsto(A,f^{-1}\circ\mu\circ ( f\otimes f), f\circ\alpha\circ f^{-1})             
\end{array}$$
defined for $x,y \in A$,  by
\begin{equation}
f\cdot\mu(x, y)=f^{-1}\mu(f(x), f(y)) , 
\quad 
f\cdot\alpha(x)=f^{-1}\alpha(f(x)). 
\end{equation}
The conjugate class is given by  $\Theta(f, (A,\mu, \alpha))=(A,f^{-1}\circ\mu\circ ( f\otimes f), f\circ\alpha\circ f^{-1}) )$ for $f\in GL_n(\mathbb{K}).$

The orbit of a  Hom-associative algebra $A$ of  $\mathcal{HA}ss_n$ is given by 
$$
\vartheta(A)=\left\{A'=f\cdot A, \, f\in GL_n(\mathbb{K})\right\}.
$$ The orbits are in \textbf{1-1 correspondence} with the isomorphism classes of $n$-dimensional Hom-associative algebras.

The stabilizer is 
$$Stab((A,\mu, \alpha))=\left\{f\in GL_n(\mathbb{K})|(f^{-1}\circ\mu\circ( f\otimes f)=\mu \text{ and  }Êf\circ \alpha =\alpha \circ f\right\}.$$
We characterize in terms of structure constants the fact  that two Hom-associative algebras are in the same  orbit (or isomorphic).
Let $(A, \mu_1, \alpha_1)$ and $(A, \mu_2, \alpha_2)$ be two $n$-dimensional Hom-associative algebras. They are isomorphic if there exists
$\varphi\in GL_n(\mathbb{K})$ such that
\begin{equation} \label{c.5}
\varphi\circ\mu_1=\mu_2(\varphi\otimes\varphi)\quad \text{and} \quad \varphi\circ\alpha_1=\alpha_2\circ\varphi.  
 \end{equation}
 \begin{remark}
 Conditions (\ref{c.5}) are equivalent to
$\mu_1=\varphi^{-1}\circ\mu_2\circ\varphi\otimes\varphi$ and $\alpha_1=\varphi^{-1}\circ\alpha_2\circ\varphi$. 
\end{remark}
\noindent We  set with respect to a basis $\left\{e_i\right\}_{i=1,\cdots,n}$:

$\varphi(e_i)=\sum_{p=1}^na_{pi}e_p,\quad\alpha_1(e_i)=\sum_{j=1}^n\alpha_{ji}e_j,\quad \alpha_2(e_i)=\sum_{j=1}^n\beta_{ji}e_j \quad i=1,\cdots, n$ 

$\mu_1(e_i, e_j)=\sum_{k=1}^n\mathcal{C}_{ij}^ke_k, \quad \mu_2(e_i, e_j)=\sum_{k=1}^n\mathcal{D}_{ij}^ke_k \quad i, j=1,\cdots, n.$

Conditions (\ref{c.5}) translate to the following
system :\\
$
\sum_{k=1}^n\mathcal{C}_{ij}^ka_{qk}-\sum_{k=1}^n\sum_{p=1}^n\mathcal{D}_{pk}^qa_{pi}a_{kj}=0,\, 
\text{and},\,\sum_{k=1}^n\alpha_{ji}a_{qk}-\sum_{k=1}^na_{ki}\beta_{qk}=0,\quad i,j, q=1,\cdots ,n.
$
\subsection{Algebraic Variety $\mathcal{HA}ss_2$}\ \\
A Hom-associative algebra is identified to its structure constants $(C_{i,j}^k)$ and $(a_{ij})$ with respect to a given basis. They satisfy the first family of system  \eqref{s1}, for which the solutions belong to the algebraic variety defined by the  following Groebner basis.
\begin{small}
\begin{align*}
\begin{array}{c} 
\begin{array}{ll} 
\langle a_{21}c_{11}^1c_{12}^1-a_{21}c_{11}^1c_{21}^1-a_{11}c_{11}^2c_{12}^1+a_{11}c_{11}^2c_{21}^1,
a_{21}c_{11}^1c_{12}^2-a_{21}c_{11}^1c_{21}^2-a_{11}c_{11}^2c_{12}^2+a_{11}c_{11}^2c_{21}^2,\\
a_{12}(c_{11}^1)^2+a_{11}c_{11}^1c_{12}^1-a_{22}c_{11}^1c_{12}^1+a_{11}c_{12}^1c_{12}^2-a_{12}c_{11}^2c_{21}^1+a_{21}c_{12}^1 c_{21}^1-a_{22}c_{11}^2
c_{22}^1+a_{21}c_{12}^2c_{22}^1,\\
a_{12}(c_{11}^2)^2+a_{12}c_{11}^2c_{12}^1-a_{11}c_{11}^1c_{21}^1+a_{22}c_{11}^1c_{21}^1-a_{21}c_{12}^1c_{21}^1-a_{11}c_{21}^1c_{21}^2+
a_{22}c_{11}^2c_{22}^1-a_{21}c_{21}^2c_{22}^1,\\
-a_{11}c_{11}^1c_{12}^1-a_{21}(c_{12}^1)^2+a_{11}c_{11}^1c_{21}^1-a_{11}c_{12}^2c_{21}^1+a_{21}(c_{21}^1)^2+a_{11}c_{12}^1c_{21}^2-a_{21}c_{12}^2
c_{22}^1+a_{21}c_{21}^2c_{22}^1,\\
a_{12}c_{11}^1c_{12}^1+a_{12}c_{12}^1c_{12}^2-a_{12}c_{11}^1c_{21}^1-a_{12}c_{21}^1c_{21}^2+a_{22}c_{12}^2c_{22}^1-a_{22}c_{21}^2c_{22}^1,
\\
-a_{22}c_{12}^1c_{21}^1+a_{12}c_{12}^1c_{22}^2+a_{22}c_{21}^1c_{22}^1-a_{12}c_{21}^1c_{22}^2,
a_{22}c_{12}^2c_{22}^1-a_{12}c_{12}^2c_{22}^2-a_{22}c_{12}^2c_{22}^2+a_{22}c_{21}^2c_{22}^2,\\
a_{12}c_{11}^1c_{11}^2+a_{11}c_{11}^2c_{12}^1-a_{22}c_{11}^1c_{12}^2+a_{11}(c_{12}^2)^2-a_{12}c_{11}^2c_{21}^2+a_{21}
c_{12}^1c_{21}^2-a_{22}c_{11}^2c_{22}^2+a_{21}c_{12}^2c_{22}^2,\\
a_{12}c_{11}^1c_{11}^2+a_{12}c_{11}^2c_{12}^2-a_{11}c_{11}^2c_{21}^1-a_{21} c_{12}^2 c_{21}^1+a_{22}c_{11}^1
c_{21}^2-a_{11}(c_{21}^2)^2+a_{22}c_{11}^2c_{22}^2-a_{21}c_{21}^2c_{22}^2,\\
-a_{11}c_{11}^2c_{12}^1-a_{21}c_{12}^1c_{12}^2+a_{11}c_{11}^2c_{21}^1+a_{21}c_{21}^1c_{21}^2-a_{21}c_{12}^2c_{22}^2+a_{21}c_{21}^2 c_{22}^2,\\
a_{12}c_{11}^2c_{12}^1+a_{12}(c_{12}^2)^2-a_{12}c_{11}^2c_{21}^1-a_{22}c_{12}^2c_{21}^1+a_{22}c_{12}^1
c_{21}^2-a_{12}(c_{21}^2)^2+a_{22}c_{12}^2c_{22}^2-a_{22}c_{21}^2 c_{22}^2,\\
a_{12}c_{11}^1c_{21}^1+a_{22}(c_{21}^1)^2+a_{12}c_{12}^1c_{21}^2-a_{11}c_{11}^1c_{22}^1-a_{21}c_{12}^1c_{22}^1+a_{22} 
c_{21}^2c_{22}^1-a_{11}c_{21}^1c_{22}^2-a_{21}c_{22}^1c_{22}^2,\\
-a_{12}c_{11}^1c_{12}^1-a_{22}(c_{12}^1)^2-a_{12}c_{12}^2c_{21}^1+a_{11}c_{11}^1c_{22}^1-a_{22}c_{12}^2c_{22}^1+a_{21}c_{21}^1
c_{22}^1+a_{11}c_{12}^1c_{22}^2+a_{21}c_{22}^1c_{22}^2,\\
a_{12}c_{11}^2c_{21}^1+a_{12}c_{12}^2c_{21}^2+a_{22}c_{21}^1c_{21}^2-a_{11}c_{11}^2c_{22}^1-a_{21}c_{12}^2
c_{22}^1-a_{11}c_{21}^2 c_{22}^2+a_{22}c_{21}^2c_{22}^2-a_{21}(c_{22}^2)^2,\\
-a_{12}c_{11}^2c_{12}^1-a_{22}c_{12}^1c_{12}^2-a_{12}c_{12}^2c_{21}^2+a_{11}c_{11}^2c_{22}^1+a_{21}c_{21}^2c_{22}^1+
a_{11}c_{12}^2c_{22}^2-a_{22}c_{12}^2c_{22}^2+a_{21}(c_{22}^2)^2\rangle
\end{array}
\end{array}
\end{align*}
\end{small}
If the Hom-associative algebra is multiplicative, it should satisfy further the second family of  \eqref{s1}, that is, it  belongs to the intersection with the the algebraic variety defined by the following Groebner basis.
\begin{small}
\begin{align*}
\begin{array}{c} 
\begin{array}{ll} 
\langle a_{11}c_{11}^1-a_{11}^2 c_{11}^1+a_{12}c_{11}^2-a_{11}a_{21}c_{12}^1-a_{11}a_{21}c_{21}^1-a_{21}^2c_{22}^1,\\
a_{11}a_{12}c_{11}^1+a_{11}c_{12}^1-a_{11} a_{2,2} c_{12}^1+a_{12}c_{12}^2-a_{12}a_{21}c_{21}^1-a_{21}a_{22}c_{22}^1\\
a_{11}a_{12}c_{11}^1-a_{12}a_{21}c_{12}^1+a^{11}c_{21}^1-a_{11}a_{22}c_{21}^1+a_{12}c_{21}^2-a_{21}a_{22}c_{22}^1,\\
a_{12}^2c_{11}^1-a_{12}a_{22}c_{12}^1-a_{12}a_{22}c_{21}^1+a_{11}c_{22}^1-a_{22}^2c_{22}^1+a_{12}c_{22}^2,\\
a_{21}c_{11}^1-a_{11}^2c_{11}^2+a_{22}c_{11}^2-a_{11}a_{21}c_{12}^2-a_{11}a_{21}c_{21}^2-a_{21}^2 c_{22}^2,\\
a_{11}a_{12}c_{11}^2+a_{21}c_{12}^1+a_{22}c_{12}^2-a_{11}a_{22}c_{12}^2-a_{12}a_{21}c_{21}^2-a_{21}a_{22}c_{22}^2,\\
a_{11}a_{12}c_{11}^2-a_{12}a_{21}c_{12}^2+a_{21}c_{21}^1+a_{22}c_{21}^2-a_{11}a_{22}c_{21}^2-a_{21}a_{22}c_{22}^2,\\
a_{12}^2c_{11}^2-a_{12}a_{22}c_{12}^2-a_{12}a_{22}c_{21}^2+a_{21}c_{22}^1+a_{22}c_{22}^2-a_{22}^2c_{22}^2\rangle.
\end{array}
\end{array}
\end{align*}
\end{small}

Describing the algebraic varieties by solving such systems lead to  the  $2$-dimensional and $3$-dimensional Hom-associative algebras classifications.
\subsection{Classification of $2$-dimensional Hom-associative algebras} \ \\
We have to consider two classes of morphisms which are given by Jordan forms, namely they are represented by the matrices
$\left(\begin{array}{ccc}
a&0\\
0&b
\end{array}
\right)$
 and 
 $\left(\begin{array}{ccc}
a&1\\
0&a
\end{array}
\right)$. 
We check whether the previous are isomorphic. We  provide all $2$-dimensional Hom-associative algebras, corresponding to solutions of the system \eqref{s1}. To this end, we use a computer algebra system.

\begin{lemma}\label{2l}
Let $\alpha$ be a diagonal morphism such that $\alpha(e_1)=pe_1, \, \alpha(e_2)=qe_2, \, p\neq q$ with respect to basis 
$\left\{e_1, e_2\right\}$. Then any $\varphi : A\rightarrow A$ such that $\varphi\circ \alpha=\alpha\circ \varphi$ is of the form 
$\varphi(e_1)=\lambda e_1$ and $\varphi(e_2)=\rho e_2$ with respect to the same basis.
\end{lemma}
\begin{proof}
Let $\varphi(e_1)=\lambda_1e_1+\lambda_2e_2$ and $\varphi(e_2)=\rho_1e_1+\rho_2e_2$. On the one hand,
$\varphi\circ\alpha(e_1)=\lambda_1pe_1+\lambda_2pe_2$ and $\alpha'\circ\varphi(e_1)=\lambda_1p'e_1+\lambda_2q'e_2$. So we have
$\lambda_1p=\lambda_1p'$ and $\lambda_2p=\lambda_2q'$. On the other hand, $\varphi\circ\alpha(e_2)=q\rho_1e_1+q\rho_2e_2$ and 
$\alpha'\circ\varphi(e_2)=\rho_1p'e_1+\rho_2q'e_2$. We have $\rho_1q=\rho_1p'$ and $\rho_2q=\rho_2q'$.\\ Then we have 
$ \lambda_1(p-p')=0, \quad \lambda_2(p-q')=0, \quad \rho_1(q-p')=0, \quad \rho_2(q-q')=0.$
If $p=p'$ and $q=q'$, we have $\lambda_2(p-q')=0$ and $\rho_2(q-q')=0.$\\
If $p\neq q$, so $\lambda_2=\rho_1=0$. Hence the lemma with $\lambda=\lambda_1$ and $\rho=\rho_2$.
\end{proof}  
	\begin{theorem}
Every $2$-dimensional multiplicative Hom-associative algebra is isomorphic to one of the following pairwise non-isomorphic Hom-associative algebra $(A, \ast,\alpha)$, where $\ast$ is the multiplication and $\alpha$ the structure map. We set $\left\{e_1, e_2\right\}$ to be a basis of 
$\mathbb{K}^2$. 
\end{theorem}
\begin{itemize}
\item [$A^2_1$] : $e_1\ast e_1=-e_1,\quad e_1\ast e_2=e_2,\quad e_2\ast e_1=e_2,\quad e_2\ast e_2=e_1,\quad\alpha(e_1)=e_1,\quad\alpha(e_2)=-e_2;$
	\item [$A^2_2$] : $e_1\ast e_1=e_1,\quad e_1\ast e_2=0,\quad e_2\ast e_1=0,\quad e_2\ast e_2=e_2,\quad \alpha(e_1)=e_1,\quad \alpha(e_2)=0;$ 
	\item [$A^2_3$] : $e_1\ast e_1=e_1,\quad e_1\ast e_2=0,\quad e_2\ast e_1=0,\quad e_2\ast e_2=0,\quad \alpha(e_1)=e_1,\quad \alpha(e_2)=0;$
\item [$A^2_4$] : $e_1\ast e_1=e_1,\quad e_1\ast e_2=e_2,\quad e_2\ast e_1=e_2,\quad e_2\ast e_2=0,\quad \alpha(e_1)=e_1,\quad \alpha(e_2)=e_2;$
\item [$A^2_5$] : $e_1\ast e_1=e_1,\quad e_1\ast e_2=0,\quad e_2\ast e_1=0,\quad e_2\ast e_2=0,\quad \alpha(e_1)=0,\quad \alpha(e_2)=ke_2;$
\item [$A^2_6$] : $e_1\ast e_1=e_2,\quad e_1\ast e_2=0,\quad e_2\ast e_1=0,\quad e_2\ast e_2=0,\quad \alpha(e_1)=e_1,\quad \alpha(e_2)=e_2;$
\item [$A^2_7$] : $e_1\ast e_1=0,\quad e_1\ast e_2=ae_1,\quad e_2\ast e_1=be_1,\quad e_2\ast e_2=ce_1,\quad \alpha(e_1)=0,\quad\alpha(e_2)=e_1,
$ where $a, b, c, k\in \mathbb{C};$
\item [$A^2_8$] : $e_1\ast e_1=0,\quad e_1\ast e_2=e_1,\quad e_2\ast e_1=0,\quad e_2\ast e_2=e_1+e_2,\quad \alpha(e_1)=e_1,\quad \alpha(e_2)=e_1+e2_;$
\item [$A^2_9$] : $e_1\ast e_1=0,\quad e_1\ast e_2=0,\quad e_2\ast e_1=e_1,\quad e_2\ast e_2=e_1+e_2,\quad \alpha(e_1)=e_1,\quad \alpha(e_2)=e_1+e_2.$
\end{itemize}
\begin{proof}
The proof follows from straightforward calculation using Definition \ref{d2} and Lemma \ref{2l}.
\end{proof}
\begin{proposition}
The Hom-associative algebras $A_1, A_4, A_6, A_8, A_9$ are of associative type.
\end{proposition}
\begin{proof}
Indeed, we set in the following  corresponding associative algebras : 
\begin{enumerate}
\item[$\tilde{A}^2_1$] : $e_1\cdot e_1=-e_1,\quad e_1\cdot e_2=-e_2,\quad e_2\cdot e_1=-e_2, \quad e_2\cdot e_2=e_1.$
\item[$\tilde{A}^2_4$] : $e_1\cdot e_1=e_1,\quad e_1 \cdot e_2=e_2,\quad e_2 \cdot e_1=e_2,\quad e_2\cdot e_2=0.$
\item[$\tilde{A}^2_6$] : $e_1\cdot e_1=e_2,\quad e_1 \cdot e_2=0,\quad e_2 \cdot e_1=0,\quad e_2\cdot e_2=0.$
\item[$\tilde{A}^2_8$] : $e_1\cdot e_1=0,\quad e_1 \cdot e_2=e_1,\quad e_2 \cdot e_1=0,\quad e_2\cdot e_2=e_2.$
\item[$\tilde{A}^2_9$] : $e_1\cdot e_1=0,\quad e_1 \cdot e_2=e_1,\quad e_2 \cdot e_1=0,\quad e_2\cdot e_2=e_2.$
\end{enumerate}
\end{proof}
\begin{remark}
It turns out that $A^2_2, A^2_3, A^2_5, A^2_7$ cannot be obtained by twisting of an associative algebra.
\end{remark}
\subsection{Classification of $3$-dimensional Hom-associative algebras}\ \\
We seek for all $3$-dimensional Hom-associative algebras. 
 We consider two classes of morphism which are given by Jordan form,  namely they are represented by the matrices
$$\left(\begin{array}{ccc}
a&0&0\\
0&b&0\\
0&0&c
\end{array}
\right), \quad \left(\begin{array}{ccc}
a&1&0\\
0&a&0\\
0&0&b
\end{array}
\right), \quad
\left(\begin{array}{ccc}
a&1&0\\
0&a&1\\
0&0&a
\end{array}
\right).$$
Using similar calculation as in previous section, we obtain the following classification. 
\begin{theorem} 
  Every $3$-dimensional multiplicative Hom-associative algebra is isomorphic to one of the following pairwise non-isomorphic Hom-associative algebra $(A, \ast,\alpha)$, where $\ast$ is the multiplication and $\alpha$ the structure map. We set $\left\{e_1, e_2,e_3\right\}$ to be a basis of	$\mathbb{K}^3$ : (the non written  products and  images of $\alpha$ are equal to zero)
\end{theorem}
\begin{itemize}
\item [$A^3_1$] : $e_1\ast e_1=e_1,\quad e_2\ast e_2=e_2+e_3\quad e_2\ast e_3=e_2+e_3,\quad e_3\ast e_2=e_2+e_3,\, e_3\ast e_3=e_2+e_3,\quad   \alpha(e_1)=e_1;$
	\item [$A^3_2$] : $e_1\ast e_1=p_{1}e_1,\quad e_2\ast e_2=p_{2}e_2,\quad e_3\ast e_3=p_{3}e_3,\quad \alpha(e_1)=e_1,\quad 
\alpha(e_2)=e_2;$
\item [$A^3_3$] : $e_1\ast e_1=p_{1}e_1,\quad e_2\ast e_2=p_{2}e_2,\quad e_3\ast e_3=p_{3}e_3,\quad \alpha(e_1)=e_1,\quad
\alpha(e_2)=e_2,\quad\alpha(e_3)=e_3;$
\item [$A^3_4$] : $e_1\ast e_2=p_{1}e_1,\,e_1\ast e_3=p_{2}e_1,\, e_2\ast e_2=p_{3}e_1,\quad e_2\ast e_3=p_{4}e_1,\,
 e_3\ast e_1=p_{5}e_1,\quad e_3\ast e_2=p_{4}e_1,\\ e_3\ast e_3=p_{6}e_1,\quad \alpha(e_2)=e_1;$
\item [$A^3_5$] : $e_2\ast e_2=p_{1}e_1,\quad e_3\ast e_3=p_{2}e_3,\quad \alpha(e_1)=e_1,\quad \alpha(e_2)=e_1+e_2;$
\item [$A^3_6$] : $e_1\ast e_2=e_1,\quad e_2\ast e_2=e_1,\quad e_2\ast e_3=e_1,\quad e_3\ast e_2=e_1,\quad \alpha(e_2)=e_1,\quad
\alpha(e_3)=e_3;$
\item [$A^3_7$] : $e_2\ast e_2=e_1,\quad e_2\ast e_3=e_1,\,e_3\ast e_2=e_1,\quad e_3\ast e_3=e_1,\quad \alpha(e_1)=e_1,\quad 
\alpha(e_2)=e_1+e_2,\quad \alpha(e_3)=e_3;$
\item [$A^3_8$] : $e_1\ast e_2=-e_3,\quad e_2\ast e_1=e_3,\quad e_2\ast e_2=e_3,\quad \alpha(e_1)=e_1,\quad\alpha(e_2)=e_1+e_2,\quad
\alpha(e_3)=e_3;$
\item [$A^3_9$] : $e_2\ast e_3=p_{1}e_1,\quad e_3\ast e_2=p_{2}e_1,\quad \alpha(e_1)=ae_1,\quad\alpha(e_2)=e_1+ae_2,\quad
\alpha(e_3)=e_3;$
\item [$A^3_{10}$] : $e_2\ast e_2=p_{1}e_1,\quad e_3\ast e_3=p_{2}e_1,\quad \alpha(e_1)=e_1,\quad\alpha(e_2)=e_1+e_2,\quad
\alpha(e_3)=-e_3;$
\item [$A^3_{11}$] : $e_1\ast e_3=p_{1}e_1,\quad e_2\ast e_3=p_{2}e_1,\quad e_3\ast e_3=p_{3}e_1,\quad \alpha(e_2)=e_1,\quad
\alpha(e_3)=e_2;$
\item [$A^3_{12}$] : $e_2\ast e_3=-p_{1}e_1,\quad e_3\ast e_2=p_{1}e_1,\quad e_3\ast e_3=p_{2}e_1,\quad \alpha(e_1)=e_1,\quad
\alpha(e_2)=e_1+e_2,\quad\alpha(e_3)=e_2+e_3.$
\end{itemize}
\begin{proposition}
The Hom-associative algebras $A^3_3, A^3_{7}, A^3_{8}, A^3_{9}, A^3_{10}, A^3_{12}$ are of associative type.
\end{proposition}
\begin{proof}
Indeed, we set in the following the corresponding associative algebras : 
\begin{enumerate}
\item[$\tilde{A}^3_{3}$] :$e_1\cdot e_1=p_{1}e_1,\quad e_2\cdot e_2=p_{2}e_2,\quad e_3\cdot e_3=p_{3}e_3;$
\item[$\tilde{A}^3_{7}$] :$e_2\cdot e_1=e_1\quad e_2\cdot e_2=e_1,\quad e_2\cdot e_3=e_1,\quad  e_3 \cdot e_2=e_1,\quad e_3\cdot e_3=e_1;$
\item[$\tilde{A}^3_{8}$] :$e_1\cdot e_2=-e_3,\quad e_2\cdot e_1=e_3\quad e_2\cdot e_2=e_3;$
\item[$\tilde{A}^3_{9}$] :$e_2\cdot e_3=\frac{p_{1}}{a}e_1,\quad e_3 \cdot e_2=\frac{p_{2}}{a}e_1;$
\item[$\tilde{A}^3_{10}$] :$e_2\cdot e_2=p_{1}e_1,\quad e_3\cdot e_3=p_{2}e_1;$
\item[$\tilde{A}^3_{12}$] :$e_2\cdot e_1=e_3\quad e_2\cdot e_3=-p_{1}e_1,\quad e_3 \cdot e_2=p_{1}e_1,\quad e_3\cdot e_3=p_{2}e_1$;
\end{enumerate}
where $p_i$ are parameters.
\end{proof}
\begin{remark}
It turns out that $A^3_1,A^3_2,A^3_4,A^3_5,A^3_{6},A^3_{11}$ cannot be obtained by twisting of an associative algebra.
\end{remark}
\begin{theorem}
  Every $2$-dimensional unital multiplicative Hom-associative algebra is isomorphic to one of the following pairwise non-isomorphic Hom-associative algebra $(A, \ast,\alpha)$, where $\ast$ is the multiplication and $\alpha$ the structure map. We set $\left\{e_1, e_2\right\}$ 
	to be a basis of $\mathbb{K}^2$ where $e_1$ is  the unit : 
\end{theorem}
\begin{itemize}
	\item [$A'^2_1$] : $e_1\ast e_1=e_1,\quad e_1\ast e_2=e_2,\quad e_2\ast e_1=e_2,\quad e_2\ast e_2=e_1+e_2,\quad\alpha(e_1)=e_1,\quad
\alpha(e_2)=e_2;$
\item [$A'^2_2$] : $e_1\ast e_1=e_1,\quad e_1\ast e_2=-e_2,\quad e_2\ast e_1=-e_2,\quad e_2\ast e_2=e_1,\quad\alpha(e_1)=e_1,\quad
\alpha(e_2)=-e_2;$
\item [$A'^2_3$] : $e_1\ast e_1=e_1,\quad e_1 \ast e_2=0,\quad e_2\ast e_1=0,\quad e_2\ast e_2=e_2,\quad\alpha(e_1)=e_1,\quad\alpha(e_2)=0;$
\item [$A'^2_3$] : $e_1\ast e_1=e_1,\quad e_1 \ast e_2=e_2,\quad e_2\ast e_1=e_2,\quad e_2\ast e_2=0,\quad\alpha(e_1)=e_1,\quad\alpha(e_2)=e_2 .$
\end{itemize}
\begin{proposition}
The unital Hom-associative algebras  $\tilde{A'}^2_1, \tilde{A'}^2_2, \tilde{A'}^2_4$ are of associative type.
\end{proposition}
\begin{proof}
Indeed, we set in the following the corresponding associative algebras : 
\begin{enumerate}
\item[$\tilde{A'}^2_1$] : $e_1\cdot e_1=e_1,\quad e_1\cdot e_2=e_2,\quad e_2\cdot e_1=e_2, \quad e_2\cdot e_2=e_1+e_2.$
\item[$\tilde{A'}^2_2$] : $e_1\cdot e_1=e_1,\quad e_1 \cdot e_2=e_2,\quad e_2 \cdot e_1=e_2,\quad e_2\cdot e_2=e_1.$
\item[$\tilde{A'}^2_4$] : $e_1\cdot e_1=e_1,\quad e_1 \cdot e_2=e_2,\quad e_2 \cdot e_1=e_2,\quad e_2\cdot e_2=0.$
\end{enumerate}
\end{proof}
\begin{remark}
It turns out that $\tilde{A'}^2_3$ cannot be obtained by twisting of an associative algebra.
\end{remark}

\begin{theorem}
  Every $3$-dimensional unital multiplicative Hom-associative algebra is isomorphic to one of the following pairwise non-isomorphic Hom-associative algebras $(A, \ast,\alpha)$,  where $\ast$ is the multiplication and $\alpha$ the structure map. We set  $\left\{e_1, e_2,e_3\right\}$ to be a basis of $\mathbb{K}^3$ where $e_1$ is the unit  (the no written products and  images of $\alpha$ are equal to zero) :
\end{theorem}
\begin{itemize}
	\item [$A'^3_1$] : $e_1\ast e_1=e_1,\quad e_2\ast e_2=e_2+e_3,\quad e_2\ast e_3=e_2+e_3,\quad e_3\ast e_2=e_2+e_3,\quad e_3\ast e_3=e_2+e_3,\quad \alpha(e_1)=e_1;$
	\item [$A'^3_2$] : $e_1\ast e_1=e_1,\quad e_2\ast e_2=e_2,\quad e_3\ast e_1=e_3,\quad e_3\ast e_3=e_1+e_3,\quad\alpha(e_1)=e_1,\quad
	\alpha(e_3)=e_3;$
	\item [$A'^3_3$] : $e_1\ast e_1=e_1,\quad e_2\ast e_2=e_2,\quad e_1\ast e_3=-e_3,\quad e_3\ast e_1=-e_3,\quad e_3\ast e_3=e_1,\quad
	\alpha(e_1)=e_1,\quad\alpha(e_3)=-e_3;$
	\item [$A'^3_4$] : $e_1\ast e_1=e_1,\quad e_1\ast e_2=e_2,\quad e_2\ast e_1=e_2,\quad e_2\ast e_2=e_1+e_2,\quad e_3\ast e_3=e_3,\quad
	\alpha(e_1)=e_1,\quad\alpha(e_2)=e_2;$
	\item [$A'^3_5$] : $e_1\ast e_1=e_1,\quad e_1\ast e_2=-e_2,\quad e_2\ast e_1=-e_2,\quad e_2\ast e_2=e_1,\quad e_3\ast e_3=e_3,\quad
	\alpha(e_1)=e_1,\quad\alpha(e_2)=-e_2;$
	\item [$A'^3_6$] : $e_1\ast e_1=e_1,\quad e_1\ast e_2=e_2,\quad e_1\ast e_3=e_3,\quad e_2\ast e_1=e_2,\quad e_2\ast e_2=e_2,\quad 
	e_2\ast e_3=e_3,\,e_3\ast e_1=e_3,\,e_3\ast e_2=e_3,\\ e_3\ast e_3=e_2+e_3,\,\alpha(e_1)=e_1,\,\alpha(e_2)=e_2,\,\alpha(e_3)=e_3;$
	\item [$A'^3_7$] : $e_1\ast e_1=e_1,\quad e_1\ast e_2=e_2,\,e_1\ast e_3=-e_3,\quad e_2\ast e_1=e_2,\,e_2\ast e_2=-e_2,\quad e_2\ast e_3=e_3,\,
	e_3\ast e_1=-e_3,\\ e_3\ast e_2=e_3,\,e_3\ast e_3=e_2,\,\alpha(e_1)=e_1,\,\alpha(e_2)=e_2,\,\alpha(e_3)=-e_3;$
	\item [$A'^3_8$] : $e_1\ast e_1=e_1,\,e_1\ast e_2=-e_2,\,e_1\ast e_3=e_3,\,e_2\ast e_1=-e_2,\,e_2\ast e_2=e_3,\,e_2\ast e_3=e_2,\,
	e_3\ast e_1=e_3,\,e_3\ast e_2=e_2,\\e_3\ast e_3=-e_3,\,\alpha(e_1)=e_1,\,\alpha(e_2)=-e_2,\,\alpha(e_3)=e_3;$	
	\item [$A'^3_9$] : $e_1\ast e_1=e_1,\quad e_1\ast e_2=ae_2,\quad e_1\ast e_3=e_3,\quad e_2\ast e_1=ae_2,\quad e_3\ast e_1=e_3,\quad
		e_3\ast e_3=e_3,\quad\alpha(e_1)=e_1,\\ \alpha(e_2)=ae_2,\quad\alpha(e_3)=e_3;$
	\item [$A'^3_{10}$] : $e_1\ast e_1=e_1,\,e_1\ast e_2=ae_2,\,e_1\ast e_3=-e_3,\,e_2\ast e_1=ae_2,\,e_3\ast e_1=-e_3,\,
		\alpha(e_1)=e_1,\,\alpha(e_2)=ae_2,\,\alpha(e_3)=-e_3;$	
	\item [$A'^3_{11}$] : $e_1\ast e_1=e_1,\quad e_1\ast e_2=ae_2,\quad e_1\ast e_3=a^2e_3,\quad e_2\ast e_1=ae_2,\,e_2\ast e_2=e_3,\,e_3\ast e_1 =a^2e_3,\,\alpha(e_1)=e_1,\\ \alpha(e_2)=ae_2,\quad\alpha(e_3)=a^2e_3;$
	\item [$A'^3_{12}$] : $e_1\ast e_1=e_1,\,e_1\ast e_2=e_2,\,e_1\ast e_3=be_3,\,e_2\ast e_1=e_2,\,e_2\ast e_2=\frac{1}{b}e_2,\,e_2\ast e_3=e_3,\,e_3\ast e_1=be_3,\,\alpha(e_1)=e_1,\,\alpha(e_2)=e_2,\,\alpha(e_3)=be_3;$
	\item [$A'^3_{13}$] : $e_1\ast e_1=e_1,\,e_1\ast e_2=-e_2,\,e_1\ast e_3=be_3,\,e_2\ast e_1=-e_2,\,e_3\ast e_1=be_3,\,\alpha(e_1)=e_1,\, 
\alpha(e_2)=-e_2,\,\alpha(e_3)=be_3;$
 \item [$A'^3_{14}$] : $e_1\ast e_1=e_1,\,e_1\ast e_2=b^2e_2,\,e_1\ast e_3=be_3,\quad e_2\ast e_1=b^2e_2,\,e_3\ast e_1=be_3,\,e_3\ast e_3=e_2,\,
 \alpha(e_1)=e_1,\\ \alpha(e_2)=b^2e_2,\,\alpha(e_3)=be_3;$
 \item [$A'^3_{15}$] : $e_1\ast e_1=e_1,\,e_1\ast e_2=ae_2,\,e_1\ast e_3=be_3,\,e_2\ast e_1=ae_2,\,e_3\ast e_1=be_3,\,\alpha(e_1)=e_1,\, 
\alpha(e_2)=ae_2,\,\alpha(e_3)=be_3.$
\end{itemize}
\begin{propo}
The unital Hom-associative algebras 
$\tilde{A'}^3_6, \tilde{A'}^3_{7},\tilde{A'}^3_{8}, \tilde{A'}^3_{9},\tilde{A'}^3_{10},\tilde{A'}^3_{11},\tilde{A'}^3_{12},$ 
$\tilde{A'}^3_{13},\tilde{A'}^3_{14},\tilde{A'}^3_{15}$ are of associative type.
\end{propo}
\begin{proof}
Indeed, we set in the following, the corresponding associative algebras : 
\begin{enumerate}
\item[$\tilde{A'}^3_6$] :$e_1\cdot e_1=e_1,\,e_1\cdot e_2=e_2,\,e_1\cdot e_3=e_3,\,e_2\cdot e_1=e_2\,e_2\cdot e_2=e_2,\,e_2\cdot e_3=e_3,\,
 e_3\cdot e_1=e_3,\,e_3\cdot e_2=e_3,\,e_3\cdot e_3=e_2+e_3;$
\item[$\tilde{A'}^3_{7}$] :$e_1\cdot e_1=e_1,\,e_1\cdot e_2=e_2,\,e_1\cdot e_3=e_3,\quad e_2\cdot e_1=e_2\,e_2\cdot e_2=-e_2,
e_2\cdot e_3=-e_3,\,e_3 \cdot e_1=e_3,\\ e_3 \cdot e_2=-e_3,\quad e_3\cdot e_3=e_2;$
\item[$\tilde{A'}^3_{8}$] :$e_1\cdot e_1=e_1,\,e_1\cdot e_2=e_2,\,e_1\cdot e_3=e_3,\,e_2\cdot e_1=e_2\,e_2\cdot e_2=e_3,\,
e_2\cdot e_3=e_2,\,e_3 \cdot e_1=e_3,\,e_3 \cdot e_2=-e_2,\,e_3\cdot e_3=-e_3;$
\item[$\tilde{A'}^3_{9}$] :$e_1\cdot e_1=e_1,\quad e_1\cdot e_2=e_2,\quad e_1\cdot e_3=e_3,\quad e_2\cdot e_1=e_2,\quad e_3\cdot e_1=e_3,e_3\cdot e_3=e_3;$
\item[$\tilde{A'}^3_{10}$] :$e_1\cdot e_1=e_1,\quad e_1\cdot e_2=e_2,\quad e_1\cdot e_3=e_3,\quad e_2\cdot e_1=e_2,\quad e_3 \cdot e_1=e_3;$
\item[$\tilde{A'}^3_{11}$] :$e_1\cdot e_1=e_1,\quad e_1\cdot e_2=e_2,\quad e_1\cdot e_3=e_3, \quad e_2\cdot e_1=e_2\quad e_2\cdot e_2=e_3,
\quad e_3 \cdot e_1=e_3;$
\item[$\tilde{A'}^3_{12}$] :$e_1\cdot e_1=e_1,\quad e_1\cdot e_2=e_2,\quad e_1\cdot e_3=e_3,\quad e_2\cdot e_1=e_2\quad e_2\cdot e_2=e_2,\quad
 e_2\cdot e_3=e_3,\quad e_3 \cdot e_1=e_3;$
\item[$\tilde{A'}^3_{14}$] :$e_1\cdot e_1=e_1,\quad e_1\cdot e_2=e_2,\quad e_1\cdot e_3=e_3, \quad e_2\cdot e_1=e_2,\quad 
 e_3 \cdot e_1=e_3,\quad e_3\cdot e_3=e_2.$
\end{enumerate}
\end{proof}
\begin{remark}
It turns out that $\tilde{A'}^3_1,\tilde{A'}^3_2,\tilde{A'}^3_3, \tilde{A'}^3_4, \tilde{A'}^3_5$ cannot be obtained by  twisting 
an associative algebra.
\end{remark}
\section{Derivations of Hom-associative algebras}
Let $(A, \mu, \alpha)$ be a multiplicative Hom-associative algebra. For any nonnegative integer $k$, we denote by $\alpha^{k}$ the $k$-times 
composition of $\alpha$, i.e $\alpha^k=\alpha\circ\cdots\circ\alpha$ ($k$-times). In particular, $\alpha^0=id$ and $\alpha^1=\alpha$.
\begin{definition}\label{df6}
For any non-negative integer $k$, a linear map $D : A\longrightarrow A$ is called an $\alpha^k$-derivation of a  Hom-associative $(A, \mu, \alpha)$, if
\begin{equation}\label{eq6}
  D\circ\alpha=\alpha\circ D 
\end{equation}
and
\begin{equation} 
D\circ\mu(f,g)=\mu(D(f),\alpha^k(g))+\mu(\alpha^k(f),D(g)).
\end{equation}
\end{definition} 
Denote by $D er_{\alpha^k}(A)$ the set of $\alpha^k$-derivations of a multiplicative Hom-associative algebra  $(A, \mu, \alpha)$. 
For any $f\in A$ satisfying $\alpha(f)=f$, we define $D_k(f) : A\rightarrow A$ by 
\begin{equation}
D_k(f)(g)=\mu(\alpha^k(g),f),\quad \forall g\in A. \nonumber
\end{equation}
Then $D(f)$ is an $\alpha^{k+1}$-derivation, which we will call an \textbf{Inner $\alpha^{k+1}$}-derivation. In fact, we have 
$D_k(f)(\alpha(g))=\mu(\alpha^{k+1}(g),f) =\alpha(\mu(\alpha^k(g),f)=\alpha\circ D_k(f)(g)$, which implies that identity  (\ref{eq6}) in Definition \ref{df6} is satisfied. On the other hand, we have 
$$\begin{array}{ll}
D_k(f)\mu(g,h)
&=\mu(\alpha^k(\mu(g,h),f)=\mu(\mu(\alpha^k(g),\alpha^k(h)),\alpha(f))\\
&=\mu(\alpha^{k+1}(g), \mu(\alpha^{k}(h),f))+\mu(\mu(\alpha^k(g),f),\alpha^{k+1}(h))\\
&=\mu(\alpha^{k+1}(g), D_k(f)(h))+\mu(D_k(f)(g),\alpha^{k+1}(h)).
\end{array}$$ 
Therefore, $D_k(f)$ is an $\alpha^{k+1}$-derivation.  The set of $\alpha^k$-derivations is denoted by \textbf{Inner$_{\alpha^k}(A)$} , i.e.
\begin{equation}
\textbf{Inner}_{\alpha^k}(A)=\left\{\mu(\alpha^{k-1}(\bullet),f|f\in A, \alpha(f)=f\right\}. 
\end{equation}
For any $D\in Der _{\alpha^k}(A)$ and $D'\in Der _{\alpha^s}(A)$, we define their commutator $\left[D, D'\right]$ as usual : 
$\left[D,D'\right]=D\circ D'-D'\circ D$. 
\begin{proposition}
For any $D\in Der_{\alpha^k}(A)$ and  $D'\in Der_{\alpha^s}(A)$, we have
$\left[D,D'\right]\in Der_{\alpha^{k+s}}(A).$
\end{proposition}  
\begin{proof}
For any $f, g\in A,$ we have
$$\begin{array}{ll} 
& \left[D,D'\right]\mu(f,g)
=D\circ D'\mu(f,g)-D'\circ D\mu(f,g)\\
&=D(\mu(D'(f),\alpha^s(g))+\mu(\alpha^s(f),D'(g)))
-D'(\mu(D(f),\alpha^k(g))+\mu(\alpha^k(f),D(g)))\\
&=\mu(D\circ D'(f),\alpha^{k+s}(g))+\mu(\alpha^k\circ D'(f),D\circ\alpha^s(g))
+\mu(D\circ\alpha^s(f), \alpha^k\circ D'(g))+\mu(\alpha^{k+s}(f),D\circ D'(g))\\
&-\mu(D'\circ D(f),\alpha^{k+s}(g))-\mu(\alpha^s\circ D(f),D'\circ\alpha^k(g))
-\mu(D'\circ\alpha^k(f),\alpha^s\circ D(g))-\mu(\alpha^{k+s}(f),D'\circ D(g)).
\end{array}$$
Since $D$ and $D'$ satisfy $D\circ\alpha=\alpha\circ D,\quad D'\circ\alpha=\alpha\circ D'$, we obtain 
$\alpha^k\circ D'=D'\circ\alpha^k, \quad D\circ\alpha^s=\alpha^s\circ D$. Therefore, we have
\begin{equation}
\left[D,D'\right]\mu(f,g)=\mu(\alpha^{k+s}(f),\left[D,D'\right](g))+\mu(\left[D,D'\right](f), \alpha^{k+s}(g)).\nonumber
\end{equation}
Furthermore, it is straightforward to see that
$$
\left[D,D'\right]\circ\alpha=D\circ D'\circ\alpha-D'\circ D\circ\alpha
=\alpha\circ D\circ D'-\alpha\circ D'\circ D
=\alpha\circ\left[D, D'\right],
$$
which yields that $\left[D, D'\right]\in Der_{\alpha^{k+s}}(A)$.
\end{proof}
\section{Cohomology of Hom-associative algebras.}
In this  section, we deal with a  cochain complex that defines a  cohomology of multiplicative Hom-associative algebras and then compute the cohomology groups of Hom-associative algebras obtained in the 2-dimensional and 3-dimensional classifications.

Let $(A, \mu, \alpha)$ be a Hom-associative algebra, for $n\geq1$ we define a $\mathbb{K}$-vector space $C^n_{Hom}(A,A)$ of $n$-cochains as
follows : $\tilde{\varphi}\in C^n_{Hom}(A,A)$ is a $n$-linear map $\tilde{\varphi} : A^n\rightarrow A$ satisfying
$$
\alpha\circ\tilde{\varphi}(x_0,\cdots,x_{n-1})=\tilde{\varphi}(\alpha(x_0),\alpha(x_1),\cdots \alpha(x_{n-1})) \text{ for all }\, x_0, x_1, \cdots,x_{n-1}\in A. 
$$  
\begin{definition}
We call, for $n\geq 1$, $n$-coboundary operator of a Hom-associative algebra $(A, \mu, \alpha)$ the linear map 
$\delta^n_{Hom} : C^n_{Hom}(A,A)\rightarrow C^{n+1}_{Hom}(A,A)$ defined by 
\begin{small}
\begin{equation}\label{com1}
\begin{array}{ll}
&\delta^n_{Hom}\varphi(x_0,x_1,\dots,x_n)
=\mu(\alpha^{n-1}(x_0),\varphi(x_1,x_2,\dots,x_n))\\
&+\sum_{k=1}^n(-1)^k\varphi(\alpha(x_0),\alpha(x_1),\dots,\alpha(x_{k-2}),\mu(x_{k-1},x_k),\alpha(x_{k+1}),
\dots,\alpha(x_n))+(-1)^{n+1}\mu(\varphi(x_0,x_1,\dots,x_{n-1}),\alpha^{n-1}(x_n)).
\end{array}
\end{equation}
\end{small}
The space of $n$-cocycles  is defined by 
$
Z^n_{Hom}(A,A)=\left\{\varphi\in C^n_{Hom}(A,A) : \delta^n_{Hom}\varphi=0\right\},
$
and the space of $n$-coboundaries is defined by 
$
B^n_{Hom}(A,A)=\left\{\phi\in \delta^{n-1}_{Hom}\varphi : \varphi\in C^{n-1}(A,A)\right\}.
$
We call the $n^{th}$ cohomology group of the Hom-associative algebra $A$ the quotient 
$
H^n_{Hom}(A,A)=\frac{Z^n_{Hom}(A,A)}{B^n_{Hom}(A,A)}.
$
\end{definition}
In particular, a $2$-coboundary operator of Hom-associative algebra $A$ is given by the map 
$$
\delta^2_{Hom} : C_{Hom}^2(A,A)\rightarrow C_{Hom}^3(A, A), \, \varphi\mapsto\delta^2_{Hom}\varphi 
$$
defined by 
$$ 
\delta^2_{Hom}\varphi(x,y,z)
=\varphi(\alpha(x),\mu(y,z))-\varphi(\mu(x,y),\alpha(z))
+\mu(\alpha(x),\varphi(y,z))-\mu(\varphi(x,y),\alpha(z)).
$$ 

In order to compute the second cohomology group, we set for a 2-cochain $\varphi$,  $\varphi(e_i,e_j,e_k)=f^k_{ij}e_k.$ The
conditions $\delta^2_{Hom}\varphi(e_i,e_j,e_k)=0$ and $\alpha\circ\varphi(e_i,e_j)=\varphi(\alpha(e_i),\alpha(e_j))$ translate to the following 
system
$$\left\{\begin{array}{c}
\begin{array}{ll}  
\sum_{p=1}^n\sum_{q=1}^n(a_{pi}\mathcal{C}_{jk}^qf^r_{pq}-\mathcal{C}^p_{ij}a_{qk}f^r_{pq}+a_{pi}f^q_{jk}\mathcal{C}_{pq}^r-
f^p_{ij}a_{qk}\mathcal{C}_{pq}^r)=0,\, i,j,k,r=1,\dots, n.\\
\sum_{p=1}^na_{sp}f_{ij}^p-\sum_{p=1}^n\sum_{q=1}^na_{pi}a_{qj}f_{pq}^s=0,\quad i,j, s=1,\dots,n.
\end{array}
\end{array}\right.$$
Recall that  $\delta^1f(e_i,e_j)=f(e_i).e_j-f(e_i.e_j)+e_i.f(e_j)$.
\begin{remark}
The following  groups  correspond in Deformation theory to the space of obstructions to extend a deformation of order $p$ to a deformation of order $p+1$. 
A $3$-coboundary operator of Hom-associative algebra $A$ is given by a  map 
$$
\delta^3_{Hom} : C^3(A,A)\rightarrow C^4(A, A), \, \psi\mapsto\delta^3_{Hom}\psi 
$$
defined as 
$$\begin{array}{ll} 
\delta^3_{Hom}\psi(x,y,z,w)
&=\mu(\alpha^2(x),\psi(y,z,w))-\psi(\mu(x,y),\alpha(z),\alpha(w))\\
&+\psi(\alpha(x),\mu(y,z),\alpha(w))-\psi(\alpha(x),\alpha(y),\mu(z,w))
+\mu(\psi(x,y,z)),\alpha^2(w)).
\end{array}$$ 
For the computations, we set, for a 3-cochain $\psi$, $\psi(e_i,e_j,e_k,e_s)=\varphi^s_{ijk}e_s.$ 
Conditions $\delta^3_{Hom}\psi(e_i,e_j,e_k,e_s)=0$ and $\alpha\circ\psi(e_i,e_j,e_k)=\psi(\alpha(e_i),\alpha(e_j),\alpha(e_k))$ translate 
to the following system :
$$\left\{\begin{array}{c}
\begin{array}{ll}  
\sum_{p=1}^n\sum_{q=1}^n\sum_{r=1}^n(a_{pi}a_{rp}\varphi^q_{jkl}\mathcal{C}_{rq}^s-\mathcal{C}^p_{ij}a_{qk}a_{rl}\varphi^s_{pqr}+a_{pi}\mathcal{C}_{jk}^q\varphi^s_{pqr}-a_{pi}a_{qj}\mathcal{C}_{kl}^r\varphi^s_{pqr}\\+\varphi^p_{ijk}a_{ql}a_{rq}\mathcal{C}_{pr}^s)=0,
\quad i,j,k,r,s=1,\cdots, n.\\
\sum_{s=1}^n\varphi_{ijk}^sa_{ls}-\sum_{p=1}^n\sum_{q=1}^n\sum_{s=1}^na_{pi}a_{qj}a_{sk}\varphi_{pqs}^l=0,\quad i,j, k,l=1,=\cdots,n.
\end{array}
\end{array}\right.$$
The cohomology class is given by solving the equation  $\psi=\delta^2_{Hom}\varphi$, where  $\varphi$ is a 2-cochain. 
\end{remark}
\subsection{Cohomology and Obstructions in $\mathcal{HA}{ss_2}$}
\subsubsection*{Cohomology in $\mathcal{HA}{ss_2}$}
In the following, we compute the 2-cocycles $Z^2$ and the 2-cohomology group $H^2$ and then $Z^3$ and $H^3$ for the 2-dimensional and 3-dimensional Hom-associative algebras provided in the classification.  We do write only non-trivial images of basis elements.
\begin{enumerate}
\item For $A^2_1$, $A^2_4$, $A^2_8$ and $A^2_9$, the $Z^2$ is $0$-dimensional. Thus, we have	$H^2=\left\langle 0\right\rangle$.
	\item For $A^2_2$, the $Z^2$ is $1$-dimensional generated by the $2$-cocycle defined as :
	$\varphi_2(e_2,e_2)=e_2.$ Thus, we have $H^2=\left\langle 0\right\rangle$.
	\item For $A^2_3$, the $Z^2$ is $3$-dimensional generated by the $2$-cocycle generators :
	$\varphi_1(e_1,e_2)=e_2,$ $\varphi_2(e_2,e_1)=e_2,$ $ \varphi_3(e_2,e_2)=e_2.$ Thus, we have 
	$H^2=\left\langle\varphi_2,\varphi_3\right\rangle$.
	\item For $A^2_5$, the $Z^2$ is $1$-dimensional generated by the $2$-cocycle generators :
	$\varphi_1(e_1,e_1)=e_1.$ Thus, we have $H^2=\left\langle 0\right\rangle$.
	\item For $A^2_6$ the $Z^2$ is $2$-dimensional generated by the $2$-cocycle generators :
	$\varphi_1(e_1,e_1)=e_2,$ $\varphi_2(e_1,e_2)=e_2,$ $\varphi_2(e_2,e_1)=e_2.$ Thus, we have 
	$H^2=\left\langle\varphi_2\right\rangle$.
	\item For $A^2_7$, the $Z^2$ is $2$-dimensional generated by the $2$-cocycle generators :
	$\varphi_1(e_1,e_2)=e_1,\quad\varphi_2(e_2,e_1)=e_1.$ Thus, we have 
	$H^2=\left\langle\varphi_1,\varphi_2\right\rangle$.
\end{enumerate}
\subsubsection*{Obstructions spaces  of  $\mathcal{HA}ss_2$}
\begin{enumerate}
	\item For $A^2_1$, $A^2_4$, $A^2_8$, $A^2_9$, the $Z^3$ is $0$-dimensional. Thus, we have $H^3=\left\langle 0\right\rangle.$ 
	\item For $A^2_2$, the $Z^3$ is $4$-dimensional generated by the $3$-cocycle generators : \\	 
$\begin{array}{ll}  
\psi_1(e_1,e_2,e_1)=e_2,\\
\psi_1(e_1,e_2,e_2)=-e_2,\quad
\psi_2(e_2,e_1,e_2)=e_2,
\end{array}$
$\begin{array}{ll} 
\psi_3(e_2,e_2,e_1)=e_2,\\
\psi_4(e_2,e_2,e_2)=e_2. 
\end{array}$
Thus, we have $H^3=\left\langle \psi_1,\psi_2,\psi_3,\psi_4\right\rangle.$
\item For $A^2_3$, the $Z^3$ is $6$-dimensional generated by the $3$-cocycle generators :\\	
 $\begin{array}{ll}  
\psi_1(e_1,e_2,e_1)=e_2,\\
\psi_2(e_1,e_2,e_2)=e_2, 
\end{array}$	
$\begin{array}{ll}
\psi_3(e_2,e_1,e_1)=e_2,\\ 
\psi_4(e_2,e_1,e_2)=e_2,
\end{array}$
$\begin{array}{ll} 
\psi_5(e_2,e_2,e_1)=e_2,\\
\psi_6(e_2,e_2,e_2)=e_2.
\end{array}$
Thus, we have $H^3=\left\langle \psi_1,\cdots,\psi_6\right\rangle.$
\item For $A^2_5$, the $Z^3$ is $4$-dimensional generated by the $3$-cocycle generators : \\
 $\begin{array}{ll}  
\psi_1(e_1,e_1,e_1)=e_1,\\
\psi_2(e_1,e_1,e_2)=e_1,\quad
\psi_3(e_1,e_2,e_1)=e_1, 
\end{array}$
$\begin{array}{ll}  
\psi_4(e_2,e_1,e_1)=e_1,\\
\psi_4(e_2,e_1,e_2)=-e_1. 
\end{array}$
Thus, we have $H^3=\left\langle \psi_1, \psi_3,\psi_4\right\rangle.$
\item For $A^2_6$, the $Z^3$ is $3$-dimensional generated by the $3$-cocycle generators : \\
$\begin{array}{ll}  
\psi_1(e_1,e_1,e_1)=e_1,\\
\psi_1(e_2,e_1,e_1)=2e_2, \quad
\psi_2(e_1,e_1,e_1)=e_2,
\end{array}$
$\begin{array}{ll}  
\psi_3(e_1,e_1,e_2)=e_2. \\
\psi_3(e_2,e_1,e_1)=-e_2,
\end{array}$
Thus, we have $H^3=\left\langle \psi_1,\varphi_2,\varphi_3\right\rangle.$
\item For $A^2_7$, the $Z^3$ is $2$-dimensional generated by the $3$-cocycle generators : \\
$\begin{array}{ll}  
\psi_1(e_1,e_2,e_2)=e_1,\\
\psi_1(e_2,e_2,e_2)=\frac{1}{2}e_2,
\end{array}$
$\begin{array}{ll}  
\psi_2(e_2,e_2,e_1)=e_1,\quad
\psi_2(e_2,e_2,e_2)=\frac{1}{2}e_2.
\end{array}$	
Thus, we have  $H^3=\left\langle \psi_1,\psi_2\right\rangle.$ 
\end{enumerate}
\subsection{Cohomology and Obstructions in $\mathcal{HA}{ss_3}$}
\subsubsection*{Cohomology in $\mathcal{HA}{ss_3}$}
\begin{enumerate}
		\item For $A^3_1$, the $Z^2$ is $12$-dimensional.Thus, we have $H^2=\left\langle\varphi_1,\dots,\varphi_{12}\right\rangle.$
\item For $A^3_2$, the $Z^2$ is $1$-dimensional generated by the $2$-cocycle generators	$\varphi_1(e_3,e_3)=e_3$. Thus, we have 
$H^2=\left\langle\varphi_1\right\rangle.$	
\item For $A^3_3$,$A^3_9$, the $Z^2$ is $0$-dimensional. Thus, we have $H^2=\left\langle\left\{0\right\} \right\rangle.$
\item For $A^3_4$, the $Z^2$ is $7$-dimensional. Thus, we have $H^2=\left\langle\varphi_1,\dots,\varphi_7\right\rangle.$
 \item For $A^3_5$, the $Z^2$ is $1$-dimensional generated by the $2$-cocycle generators $\varphi_1(e_3,e_3)=e_3$.
 Thus, we have $H^2=\left\langle 0\right\rangle.$
  \item For $A^3_6$, the $Z^2$ is $3$-dimensional generated by the $2$-cocycle generators 
	$\varphi_1(e_1,e_2)=e_1,$ $ \varphi_2(e_2,e_1)=e_1,$ $  \varphi_3(e_3,e_3)=e_3.$	Thus, we have $H^2=\left\langle \varphi_3\right\rangle.$
  \item For $A^3_7$ , the $Z^2$ is $1$-dimensional generated by the $2$-cocycle generators $\varphi_1(e_1,e_2)=e_1.$
Thus, we have  $H^2=\left\langle\varphi_1\right\rangle.$
 \item For $A^3_8$, the $Z^2$ is $2$-dimensional generated by the $2$-cocycle generators 
$\varphi_1(e_1,e_1)=e_3,$ $\varphi_2(e_1,e_2)=e_3, $ $ \varphi_2(e_2,e_1)=-e_3.$
Thus, we have $H^2=\left\langle\varphi_1,\varphi_2\right\rangle.$
\item For $A^3_{10}$, the $Z^2$ is $1$-dimensional generated by the $2$-cocycle generators 
$\varphi_1(e_3,e_3)=e_1$. Thus, we have $H^2=\left\langle 0\right\rangle.$
 \item For $A^3_{11}$, the $Z^2$ is $2$-dimensional generated by the $2$-cocycle generators
$\varphi_1(e_1,e_3)=e_1, $ $ \varphi_2(e_3,e_1)=e_1.$ Thus, we have $H^2=\left\langle \varphi_2\right\rangle.$
\item For $A^3_{12}$ , the $Z^2$ is $0$-dimensional. Thus, we have $H^2=\left\langle0  \right\rangle.$
\end{enumerate}
\subsubsection*{Obstructions spaces  of  $\mathcal{HA}ss_3$}
\begin{enumerate}
	\item For $A^3_1$, the $Z^3$ is $42$-dimensional and  we have 
$H^3=\left\langle\psi_1,\cdots,\psi_{42}\right\rangle\backslash\left\langle\psi_{31},\psi_{32},\psi_{37},\psi_{38}\right\rangle$.
\item For $A^3_2$, the $Z^3$ is $9$-dimensional and  we have $H^3=\left\langle\psi_1,\cdots,\psi_{9}\right\rangle.$
\item For $A^3_3$, the $Z^3$ is $0$-dimensional generated by the $3$-cocycle generators. Thus, we have
 $H^3=\left\langle\left\{ 0\right\}\right\rangle.$  
\item For $A^3_4$, the $Z^3$ is $2$-dimensional and we have $H^3=\left\langle\psi_1,\dots,2\right\rangle$
\item For $A^3_5$, the $Z^3$ is $11$-dimensional and we have $H^3=\left\langle\psi_1,\cdots,\psi_{11}\right\rangle$.
\item For $A^3_6$, the $Z^3$ is $18$-dimensional and we have $H^3=\left\langle\psi_1,\psi_4,\psi_5,\psi_6,\psi_{7},\psi_8,\psi_9,
\psi_{10},\psi_{14},\psi_{18}\right\rangle$.
\item For $A^3_7$, the $Z^3$ is $14$-dimensional and we have $H^3=\left\langle\psi_1,\cdots,\psi_{14}\right\rangle$. 
\item For $A^3_8$, the $Z^3$ is $7$-dimensional and we have $H^3=\left\langle\psi_1,\cdots,\psi_7\right\rangle.$ 
\item For $A^3_9$, the $Z^3$ is $5$-dimensional generated by the following  $3$-cocycles generators :\\
$\begin{array}{ll}  
\psi_1(e_1,e_3,e_3)=e_1,\\
\psi_1(e_2,e_3,e_3)=e_2,\\
\psi_2(e_2,e_3,e_3)=e_1,
\end{array}$
$\begin{array}{ll}  
\psi_3(e_3,e_2,e_3)=e_1,\\
\psi_4(e_3,e_3,e_1)=e_1
\end{array}$
$\begin{array}{ll}  
\psi_4(e_3,e_3,e_2)=e_2,\\
\psi_5(e_3,e_3,e_2)=e_1.
\end{array}$
Thus, we have $H^3=\left\langle\psi_1,\cdots,\psi_5\right\rangle$.
\item For $A^3_{10}$, the $Z^3$ is $7$-dimensional generated by the $3$-cocycle generators :\\
$\begin{array}{ll}  
\psi_1(e_1,e_2,e_2)=e_1,\\
\psi_1(e_2,e_2,e_1)=-e_1,\\
\psi_1(e_2,e_2,e_3)=-e_3,\\
\psi_1(e_3,e_2,e_2)=e_3,
\end{array}$
$\begin{array}{ll}  
\psi_2(e_1,e_3,e_3)=e_1,\\
\psi_2(e_2,e_3,e_3)=e_2,\\
\psi_2(e_3,e_3,e_1)=-e_1,\\
\psi_2(e_3,e_3,e_2)=-e_2,
\end{array}$
$\begin{array}{ll}  
\psi_3(e_2,e_2,e_2)=e_1,\\
\psi_4(e_2,e_3,e_3)=e_1,\\
\psi_5(e_3,e_2,e_3)=e_1,\\
\psi_6(e_3,e_3,e_2)=e_1.
\end{array}$
Thus, we have $H^3=\left\langle\psi_1,\cdots,\psi_4\right\rangle$.
\item For $A^3_{11}$, the $Z^3$ is $19$-dimensional and we have $H^3=\left\langle\psi_1,\dots,\psi_{19}\right\rangle\backslash\left\langle\psi_4,\psi_{10},\psi_{13}\right\rangle.$ 
\item For $A^3_{12}$, the $Z^3$ is $7$-dimensional and we have $H^3=\left\langle\psi_1,\cdots,\psi_5\right\rangle$.
\end{enumerate}
\subsection{Cohomology of associative type algebras in   $\mathcal{HA}ss_n$}\

We compute the third cohomology of the associative algebras corresponding to Hom-associative algebras of associative type. The coboundary operator may be obtained from the coboundary operator of Hom-associative algebras by taking $\alpha$ equals to the identity map.
We have the following observation.
\begin{theorem}
Let $(A, \mu,\alpha)$ be a Hom-associative algebra of associative type where $\mu=\alpha\mu$ and $(A,\mu')$ is an associative algebra.
Let $\varphi'$ be a $n$-cocycle with respect to Hochschild cohomology of $(A,\mu')$. If $\varphi'$ satisfies
$\alpha\varphi'=\varphi'\circ(\alpha\otimes\alpha)$ then $\alpha\varphi'$ is a $n$-cocycle of $(A, \mu,\alpha)$ with respect to Hom-type Hochschild cohomology.  
\end{theorem}
\begin{proof}
Let 

$\begin{array}{ll}
\delta^n_{Ass}\tilde{\varphi}(x_0,\dots,x_n)
&=\tilde{\mu}(x_0,\tilde{\varphi}(x_1,x_2,\dots,x_n))
+\displaystyle\sum_{k=1}^n(-1)^k\tilde{\varphi}(x_0,\dots,x_{k-2},\tilde{\mu}(x_{k-1},x_k),x_{k+1},\dots,x_n)\\
&+(-1)^{n+1}\tilde{\mu}(\tilde{\varphi}(x_0,\dots,x_{n-1}),x_n).
\end{array}$

If $\tilde{\varphi}$ satisfies $$\alpha\circ\tilde{\varphi}(x_0,\dots,x_{n-1})=\tilde{\varphi}(\alpha(x_0),\dots,\alpha(x_{n-1}))$$
for $x_0,\dots,x_{n-1}\in A$,  by equation (\ref{com1}), we have 
\begin{small}
\begin{align*}
& \mu(\alpha^{n-1}(x_0),\varphi(x_1,x_2,\dots,x_n))\\& +\sum_{k=1}^n(-1)^k\varphi(\alpha(x_0),\dots,\alpha(x_{k-2}),\mu(x_{k-1},x_k),\alpha(x_{k+1}),\dots,\alpha(x_n))+(-1)^{n+1}\mu(\varphi(x_0,\dots,x_{n-1}),\alpha^{n-1}(x_n))=0.
\end{align*}
\end{small}
By  multiplication $\alpha=(\alpha_1,\dots,\alpha_{n})$, we obtain $\alpha\circ\tilde{\varphi}$ is a $n$-cocycle for  
$(A, \alpha\tilde{\mu}, \alpha)$.
\end{proof}
\subsection{Cohomology and Obstructions in $\mathcal{A}ss_2$}
\subsubsection*{Computation of cohomology in $\mathcal{A}ss_2$}
\begin{enumerate}
	\item  For $\tilde{A}^2_1$,the $Z^2$ is $4$-dimensional generated by the $2$-cocycle generators :\\
	$\begin{array}{ll}  
\tilde{\varphi}_1(e_1,e_1)=e_1,\\
\tilde{\varphi}_1(e_1,e_2)=e_2,
\end{array}$
$\begin{array}{ll}  
\tilde{\varphi}_1(e_2,e_1)=e_2,\\
\tilde{\varphi}_2(e_1,e_1)=e_2,
\end{array}$
$\begin{array}{ll}  
\tilde{\varphi}_2(e_1,e_2)=-e_1,\\
\tilde{\varphi}_2(e_2,e_1)=-e_2,
\end{array}$
$\begin{array}{ll}  
\tilde{\varphi}_3(e_2,e_2)=e_1,\\
\tilde{\varphi}_4(e_2,e_2)=e_2.
\end{array}$	
Thus, we have $H^2=\left\langle\tilde{\varphi}_1\right\rangle$.	
\item	$\tilde{A}^2_4$, the $Z^2$ is $4$-dimensional generated by the $2$-cocycle generators :\\
$\begin{array}{ll}  
\tilde{\varphi}_1(e_1,e_1)=e_1,\\
\tilde{\varphi}_1(e_1,e_2)=e_2,
\end{array}$
$\begin{array}{ll}  
\tilde{\varphi}_1(e_2,e_1)=e_2,\\
\tilde{\varphi}_2(e_1,e_1)=e_2,
\end{array}$
$\begin{array}{ll}  
\tilde{\varphi}_3(e_2,e_2)=e_1,\quad
\tilde{\varphi}_4(e_2,e_2)=e_2.
\end{array}$
Thus, we have $H^2=\left\langle\tilde{\varphi}_3\right\rangle$.
\item $\tilde{A}^2_6$, the $Z^2$ is $4$-dimensional generated by the $2$-cocycle generators :\\
$\begin{array}{ll}  
\tilde{\varphi}_1(e_1,e_1)=e_1,\\
\tilde{\varphi}_2(e_1,e_1)=e_2,
\end{array}$
$\begin{array}{ll}  
\tilde{\varphi}_3(e_1,e_2)=e_1,\\
\tilde{\varphi}_3(e_2,e_1)=e_1,\quad
\tilde{\varphi}_3(e_2,e_2)=e_2,
\end{array}$
$\begin{array}{ll}  
\tilde{\varphi}_4(e_2,e_2)=e_2,\\
\tilde{\varphi}_4(e_2,e_1)=e_2\\
\end{array}$
Thus, we have $H^2=\left\langle\tilde{\varphi}_1,\tilde{\varphi}_3, \tilde{\varphi}_4\right\rangle$.
\item $\tilde{A}^2_8$, the $Z^2$ is $4$-dimensional,  the $Z^2$ is $4$-dimensional generated by the $2$-cocycle generators :\\
$\begin{array}{ll}  
\tilde{\varphi}_1(e_1,e_1)=e_1,\\
\tilde{\varphi}_1(e_2,e_1)=e_2,
\end{array}$
$\begin{array}{ll}  
\tilde{\varphi}_2(e_1,e_2)=e_1,\\
\tilde{\varphi}_2(e_2,e_2)=e_2.
\end{array}$
Thus, we have $H^2=\left\langle\tilde{\varphi}_1,\tilde{\varphi}_2\right\rangle$.
\item For $\tilde{A}^2_9$, the $Z^2$ is $4$-dimensional generated by the $2$-cocycle generators :\\
$\begin{array}{ll}  
\tilde{\varphi}_1(e_1,e_1)=e_1,\\
\tilde{\varphi}_1(e_2,e_1)=e_2,
\end{array}$
$\begin{array}{ll}  
\tilde{\varphi}_2(e_1,e_2)=e_1,\\
\tilde{\varphi}_2(e_2,e_2)=e_2.
\end{array}$
Thus, we have $H^2=\left\langle\tilde{\varphi}_1,\tilde{\varphi}_2\right\rangle$.
\end{enumerate}
\subsubsection*{Obstructions spaces  of  $\mathcal{A}ss_2$}
\begin{enumerate}
\item For $\tilde{A}^2_1$ , the $Z^3$ is $4$-dimensional generated by the $3$-cocycle generators : \\
$\begin{array}{ll}  
\psi_1(e_1,e_1,e_2)=e_1,\\
\psi_1(e_1,e_2,e_2)=-e_2,\\
\psi_1(e_2,e_1,e_2)=e_2,\\
\psi_1(e_2,e_2,e_2)=e_1,
\end{array}$
$\begin{array}{ll}
\psi_2(e_1,e_1,e_2)=e_2,\\ 
\psi_2(e_1,e_2,e_2)=e_1,\\
\psi_2(e_2,e_1,e_2)=-e_1,\\
\psi_2(e_2,e_2,e_2)=e_2,
\end{array}$
$\begin{array}{ll}  
\psi_3(e_2,e_1,e_1)=e_1,\\
\psi_3(e_2,e_1,e_2)=e_2,\\
\psi_3(e_2,e_2,e_2)=-e_2\\
\psi_3(e_2,e_2,e_2)=e_1,
\end{array}$           
$\begin{array}{ll}  
\psi_4(e_2,e_1,e_1)=e_2,\\
\psi_4(e_2,e_1,e_2)=-e_1,\\
\psi_4(e_2,e_2,e_1)=e_1\\
\psi_4(e_2,e_2,e_2)=e_2.
\end{array}$

Thus, we have $H^3=\left\langle\tilde{\psi}_1,\dots,\tilde{\psi}_4\right\rangle$.
\item For $\tilde{A}^2_4$, the $Z^3$ is $5$-dimensional generated by the $3$-cocycle generators :\\ 
$\begin{array}{ll}  
\psi_1(e_1,e_1,e_2)=e_1,\\
\psi_1(e_1,e_2,e_2)=-e_2,\\
\psi_1(e_2,e_1,e_2)=e_2,
\end{array}$
$\begin{array}{ll}  
\psi_2(e_1,e_1,e_2)=e_2,\\
\psi_3(e_2,e_1,e_1)=e_1,\\
\psi_3(e_2,e_1,e_2)=e_2,
\end{array}$
$\begin{array}{ll}  
\psi_3(e_2,e_1,e_2)=-e_2,\\
\psi_4(e_2,e_1,e_2)=e_2,\\
\psi_5(e_2,e_2,e_2)=e_2.
\end{array}$                                                                                                                  
Thus, we have $H^3=\left\langle\tilde{\psi}_1,\cdots,\tilde{\psi}_5\right\rangle$.
\item For $\tilde{A}^2_6$, the $Z^3$ is $5$-dimensional generated by the $3$-cocycle generators :\\
$\begin{array}{ll}  
\psi_1(e_1,e_1,e_2)=e_2,\\
\psi_2(e_1,e_2,e_1)=e_1,
\end{array}$
$\begin{array}{ll}  
\psi_3(e_1,e_2,e_2)=e_1,\\
\psi_3(e_2,e_2,e_1)=-e_1,
\end{array}$
$\begin{array}{ll}  
\psi_4(e_2,e_1,e_2)=e_2,\\
\psi_5(e_2,e_2,e_2)=e_2.
\end{array}$
Thus, we have $H^3=\left\langle\tilde{\psi}_1,\cdots,\tilde{\psi}_5\right\rangle$.
\item For $\tilde{A}^2_8$, the $Z^3$ is $5$-dimensional generated by the $3$-cocycle generators :\\
$\begin{array}{ll}  
\psi_1(e_1,e_1,e_1)=e_1,\\
\psi_1(e_1,e_2,e_1)=-e_2,\\
\psi_1(e_2,e_1,e_1)=e_2,
\end{array}$
$\begin{array}{ll}  
\psi_2(e_1,e_1,e_2)=e_1,\\
\psi_2(e_1,e_2,e_1)=e_1,\\
\psi_2(e_1,e_2,e_2)=-e_2,\\
\psi_2(e_2,e_1,e_2)=e_2,
\end{array}$
$\begin{array}{ll}  
\psi_3(e_1,e_2,e_2)=e_1,\\
\psi_4(e_2,e_2,e_1)=e_1,\\
\psi_5(e_2,e_2,e_2)=e_1.
\end{array}$
Thus, we have $H^3=\left\langle\tilde{\psi}_1,\cdots,\tilde{\psi}_5\right\rangle$.
\item For $\tilde{A}^2_9$, the $Z^3$ is $5$-dimensional generated by the $3$-cocycle generators :\\
$\begin{array}{ll}  
\psi_1(e_1,e_1,e_1)=e_1,\\
\psi_1(e_1,e_2,e_1)=-e_2,\\
\psi_1(e_2,e_1,e_1)=e_2,
\end{array}$
$\begin{array}{ll}  
\psi_2(e_1,e_1,e_2)=e_1,\\
\psi_2(e_1,e_2,e_1)=e_1,\\
\psi_2(e_1,e_2,e_2)=-e_2,\\
\psi_2(e_2,e_1,e_2)=e_2,
\end{array}$
$\begin{array}{ll}  
\psi_3(e_1,e_2,e_2)=e_1,\\
\psi_4(e_2,e_2,e_1)=e_1,\\
\psi_5(e_2,e_2,e_2)=e_1.
\end{array}$
Thus, we have $H^3=\left\langle\tilde{\psi}_1,\cdots,\tilde{\psi}_5\right\rangle$.
\end{enumerate}
\subsection{Cohomology and Obstructions in $\mathcal{A}ss_3$}
\subsubsection*{Computation of cohomology in $\mathcal{A}ss_3$}
\begin{enumerate}
\item For $\tilde{A}^3_3$ the $Z^2$ is $9$-dimensional and we have 
	 $H^2=\left\langle\tilde{\varphi}_1,\cdots,\tilde{\varphi}_9\right\rangle$.  
\item For $\tilde{A}^3_7$ the $Z^2$ is $4$-dimensional and we have 
	 $H^2=\left\langle\tilde{\varphi}_1,\cdots,\tilde{\varphi}_4\right\rangle$.  
\item For $\tilde{A}^3_8$, the $Z^2$ is $8$-dimensional and we have $H^2=\left\langle\tilde{\varphi}_1,\cdots,\tilde{\varphi}_8\right\rangle$.  
\item For $\tilde{A}^3_9$, the $Z^2$ is $10$-dimensional and we have $H^2=\left\langle\tilde{\varphi}_1,\cdots,\tilde{\varphi}_{10}\right\rangle\backslash\left\langle\tilde{\varphi}_3,\tilde{\varphi}_6,\tilde{\varphi}_8,\tilde{\varphi}_9\right\rangle$.
\item For $\tilde{A}^3_{10}$, the $Z^2$ is $10$-dimensional and we have
 $H^2=\left\langle\tilde{\varphi}_1,\cdots,\tilde{\varphi}_{10}\right\rangle\backslash\left\langle\tilde{\varphi}_3,\tilde{\varphi}_6,\tilde{\varphi}_8,\tilde{\varphi}_9\right\rangle$.  
\item For $\tilde{A}^3_{12}$, the $Z^2$ is $2$-dimensional generated by the $2$-cocycle generators :\\ 
$\begin{array}{ll}  
\tilde{\varphi}_1(e_1,e_2)=e_1,\\
\tilde{\varphi}_1(e_2,e_1)=e_1,
\end{array}$
$\begin{array}{ll} 
\tilde{\varphi}_1(e_2,e_2)=2e_2,\\ 
\tilde{\varphi}_1(e_3,e_2)=e_3,
\end{array}$
$\begin{array}{ll} 
\tilde{\varphi}_2(e_2,e_1)=e_3,\\
\tilde{\varphi}_2(e_2,e_3)=e_1
\end{array}$
$\begin{array}{ll}  
\tilde{\varphi}_2(e_3,e_2)=-e_1,\\
\tilde{\varphi}_2(e_3,e_3)=-e_1.
\end{array}$
We have $H^2=\left\langle\tilde{\varphi}_1,\tilde{\varphi}_2\right\rangle$.  
\end{enumerate}
\subsubsection*{Obstructions spaces  of  $\mathcal{A}ss_3$}
\begin{enumerate}
\item For $\tilde{A}^3_3$, the $Z^3$ is $18$-dimensional generated by the $3$-cocycle generators.

Thus, we have $H^3=\left\langle\tilde{\psi}_1,\dots,\tilde{\psi}_{18}\right\rangle$. 
\item For $\tilde{A}^3_7$, the $Z^3$ is $5$-dimensional generated by the $3$-cocycle generators.

Thus, we have $H^3=\left\langle\tilde{\psi}_1,\cdots,\tilde{\psi}_5\right\rangle$. 
\item For $\tilde{A}^3_8$, the $Z^3$ is $24$-dimensional generated by the $3$-cocycle generators.

Thus, we have  $H^3=\left\langle\begin{array}{ll}\tilde{\psi_1},\cdots,\tilde{\psi_{24}}\end{array}\right\rangle$. 
\item For $\tilde{A}^3_9$, the $Z^3$ is $27$-dimensional generated by the $3$-cocycle generators.

Thus, we have  $H^3=\left\langle\begin{array}{ll}\tilde{\psi_1},\cdots,\tilde{\psi_{27}}\end{array}\right\rangle$. 
\item For $\tilde{A}^3_{10}$, the $Z^3$ is $23$-dimensional generated by the $3$-cocycle generators.

Thus, we have  $H^3=\left\langle\begin{array}{ll}\tilde{\psi_1},\cdots,\tilde{\psi_{23}}\end{array}\right\rangle$. 
\item For $\tilde{A}^3_{12}$, the $Z^3$ is $2$-dimensional generated by the $3$-cocycle generators :\\
$\begin{array}{ll}  
\psi_1(e_1,e_2,e_2)=e_1,\\
\psi_1(e_2,e_2,e_1)=-e_1,\\
\psi_1(e_3,e_2,e_2)=e_3,
\end{array}$
$\begin{array}{ll}  
\psi_2(e_2,e_1,e_2)=e_1,\\
\psi_2(e_2,e_1,e_3)=e_1,\\
\psi_2(e_2,e_2,e_1)=e_1,\\
\psi_2(e_2,e_2,e_3)=e_2,
\end{array}$
$\begin{array}{ll}  
\psi_2(e_2,e_3,e_2)=-e_3,\\
\psi_2(e_2,e_3,e_3)=-e_3,\\
\psi_2(e_3,e_2,e_1)=e_1.
\end{array}$
Thus, we have  $H^3=\left\langle\begin{array}{ll}\tilde{\psi_1},\tilde{\psi_2}\end{array}\right\rangle$. 
\end{enumerate}
\subsection{Cohomology and Obstructions in $\mathcal{UHA}ss_2$}
\subsubsection*{Cohomology in $\mathcal{UHA}ss_2$}
\begin{enumerate}
	\item For $A'^2_1$, $A'^2_2$ and $A'^2_4$ the $Z^2$ is $0$-dimensional and we have 
	$H^2=\left\langle0\right\rangle$. 
\item For $A'^2_3$, the $Z^2$ is $1$-dimensional generated by the $2$-cocycle generators $\varphi(e_2,e_2)=e_2$. 
We have $H^2=\left\langle 0\right\rangle$.  
\end{enumerate}
\subsubsection*{Obstructions in $\mathcal{UHA}ss_2$}
\begin{enumerate}
	\item For $A'^2_1$, $A'^2_2$ and $A'^2_4$ the $Z^3$ is $0$-dimensional and we have 
	$H^2=\left\langle0\right\rangle$.
\item For $A^2_3$, the $Z^3$ is $4$-dimensional generated by the $3$-cocycle generators :\\
$\begin{array}{ll}  
\psi'_1(e_1,e_2,e_1)=e_2,\\
\psi'_1(e_1,e_2,e_2)=-e_2,
\end{array}$
$\begin{array}{ll}  
\psi'_2(e_2,e_1,e_2)=e_2,\\
\psi'_3(e_2,e_2,e_1)=e_2,\quad
\psi'_4(e_2,e_2,e_2)=e_2.
\end{array}$
Thus, we have  $H^3=\left\langle\tilde{\psi_1},\tilde{\psi_2}\right\rangle$. 
\end{enumerate}
\subsection{Cohomology and Obstructions in $\mathcal{UHA}ss_3$}
\subsubsection*{Computation of cohomology in $\mathcal{UHA}ss_3$}
\begin{enumerate}
\item For $A'^3_1$, the $Z^2$ is $12$-dimensional generated by the $2$-cocycle generators :\\
\noindent$\begin{array}{ll}  
\varphi'_1(e_1,e_2)=e_2,\\
\varphi'_1(e_1,e_3)=-e_2,\\
\varphi'_2(e_1,e_2)=e_3,\\
\varphi'_2(e_1,e_3)=-e_3,
\end{array}$
$\begin{array}{ll}  
\varphi'_3(e_2,e_1)=e_2,\\
\varphi'_3(e_3,e_1)=-e_2,\\
\varphi'_4(e_2,e_1)=e_3,\\
\varphi'_4(e_3,e_1)=e_3,
\end{array}$
$\begin{array}{ll}  
\varphi'_5(e_2,e_2)=e_2,\\
\varphi'_6(e_2,e_2)=e_3,\\
\varphi'_7(e_2,e_3)=e_2,\\
\varphi'_8(e_2,e_3)=e_3,
\end{array}$
$\begin{array}{ll}  
\varphi'_9(e_3,e_2)=e_2,\\
\varphi'_{10}(e_3,e_2)=e_2,\\
\varphi'_{11}(e_3,e_3)=e_2,\\
\varphi'_{12}(e_3,e_3)=-e_3.
\end{array}$
We have $H^2=\left\langle\varphi'_1,\dots,\varphi'_{12}\right\rangle$.
\item For $A'^3_2$, the $Z^2$ is $1$-dimensional generated by the $2$-cocycle generators $\varphi'(e_2,e_2)=e_2$.
Thus, we have $H^2=\left\langle\varphi'_1\right\rangle$.
\item For $A'^3_3$, the $Z^2$ is $1$-dimensional generated by the $2$-cocycle generators $\varphi'(e_2,e_2)=e_2$.
Thus, we have $H^2=\left\langle 0\right\rangle$.
\item For $A'^3_4$, the $Z^2$ is $1$-dimensional generated by the $2$-cocycle generators $\varphi'(e_3,e_3)=e_3$.
Thus, we have $H^2=\left\langle \varphi'_1\right\rangle$.
\item For $A'^3_5$, the $Z^2$ is $1$-dimensional generated by the $2$-cocycle generators $\varphi'(e_3,e_3)=e_2$.
Thus, we have $H^2=\left\langle\varphi'_1\right\rangle$.
\item For $A'^3_6$, $A'^3_7$, $A'^3_8$, $A'^3_9$, $A'^3_{10}$, $A'^3_{11}$, $A'^3_{12}$, $A'^3_{13}$, $A'^3_{14}$, $A'^3_{15}$  the $Z^2$ is $0$-dimensional generated by the $2$-cocycle generators $\varphi'=0$.
\end{enumerate}
\subsubsection*{Obstructions Spaces in $\mathcal{UHA}ss_3$}
\begin{enumerate}
	\item For $A'^3_{1}$, the $Z^3$ is $42$-dimensional and  we have 
$H^3=\left\langle\psi_1,\cdots,\psi_{42}\right\rangle\backslash\left\langle\psi_{31},\psi_{32},\psi_{37},\psi_{38}\right\rangle$.
\item For $A'^3_{2}$, the $Z^3$ is $11$-dimensional.
Thus, we have $H^3=\left\langle\psi'_1,\dots,\psi'_{11}\right\rangle\backslash\left\langle\psi'_4,\psi'_5\right\rangle$.  
\item For $A'^3_{3}$, the $Z^3$ is $11$-dimensional.
Thus, we have $H^3=\left\langle\psi'_1,\dots,\psi'_{11}\right\rangle\backslash\left\langle\psi'_2,\psi'_7,\psi'_8,\psi'_{10}\right\rangle$. 
\item For $A'^3_{4}$, the $Z^3$ is $11$-dimensional.
Thus, we have $H^3=\left\langle\psi'_1,\dots,\psi'_{11}\right\rangle\backslash\left\langle\psi'_3,\psi'_8,\psi'_{10}\right\rangle$. 
\item For $A'^3_{5}$, the $Z^3$ is $11$-dimensional.
Thus, we have $H^3=\left\langle\psi'_1,\dots,\psi'_{11}\right\rangle\backslash\left\langle\psi'_9,\psi'_{10}\right\rangle$.
\item For $A'^3_{6}$,$A'^3_{7}$,$A'^3_{8}$, $A'^3_{9}$,$A'^3_{10}$,$A'^3_{11}$,$A'^3_{12}$,$A'^3_{13}$,$A'^3_{14}$, $A'^3_{15}$, the $Z^3$ is $0$-dimensional. Thus, we have $H^3=\left\langle0\right\rangle$. 
\end{enumerate}
\subsection{Cohomology and Obstructions in $\mathcal{UA}ss_2$}
\subsubsection*{Computation of cohomology in $\mathcal{UA}ss_2$}
\begin{enumerate}
	\item For $\tilde{A}'^2_1$, the $Z^2$ is $5$-dimensional generated by the $2$-cocycle generators : \\	
	$\begin{array}{ll}  
\tilde{\varphi'}_1(e_1,e_1)=e_1,\\
\tilde{\varphi'}_1(e_1,e_2)=e_1,\\
\tilde{\varphi'}_1(e_2,e_1)=e_2,
\end{array}$
$\begin{array}{ll} 
\tilde{\varphi'}_2(e_1,e_1)=e_1+e_2,\\ 
\tilde{\varphi'}_2(e_2,e_1)=e_1+e_2,
\end{array}$
$\begin{array}{ll} 
\tilde{\varphi'}_3(e_2,e_2)=e_1,\\
\tilde{\varphi'}_4(e_2,e_2)=e_2.
\end{array}$	
Thus, we have $H^2=\left\langle\tilde{\varphi'_1},\tilde{\varphi'_2},\tilde{\varphi'_3}\right\rangle$.
\item For $\tilde{A}'^2_2$, the $Z^2$ is $4$-dimensional generated by the $2$-cocycle generators :\\
	$\begin{array}{ll}  
\tilde{\varphi'}_1(e_1,e_1)=e_1,\\
\tilde{\varphi'}_1(e_1,e_2)=e_2,\\
\tilde{\varphi'}_1(e_2,e_1)=e_2,
\end{array}$
$\begin{array}{ll} 
\tilde{\varphi'}_2(e_1,e_1)=e_2,\\ 
\tilde{\varphi'}_2(e_1,e_2)=e_1,\\
\tilde{\varphi'}_2(e_2,e_1)=e_2,
\end{array}$
$\begin{array}{ll} 
\tilde{\varphi'}_3(e_2,e_2)=e_2,\\
\tilde{\varphi'}_4(e_2,e_2)=e_2.
\end{array}$
 Thus, we have $H^2=\left\langle\tilde{\varphi'_1},\tilde{\varphi'_2}\right\rangle$.
\item For $\tilde{A}'^2_4$, the $Z^2$ is $4$-dimensional generated by the $2$-cocycle generators :\\
	$\begin{array}{ll}  
\tilde{\varphi'}_1(e_1,e_1)=e_1,\\
\tilde{\varphi'}_1(e_1,e_2)=e_2,
\end{array}$
$\begin{array}{ll} 
\tilde{\varphi'}_1(e_2,e_1)=e_2,\\ 
\tilde{\varphi'}_2(e_1,e_1)=e_2,
\end{array}$
$\begin{array}{ll} 
\tilde{\varphi'}_3(e_2,e_2)=e_1,\\
\tilde{\varphi'}_4(e_2,e_2)=e_2.
\end{array}$
 Thus, we have $H^2=\left\langle\varphi'_2,\varphi'_3\right\rangle$.
\end{enumerate}
\subsubsection*{Obstructions spaces in $\mathcal{UA}ss_2$}
\begin{enumerate}
	\item For $\tilde{A}'^2_{1}$, the $Z^3$ is $1$-dimensional generated by the $3$-cocycle generators :\\
	$\tilde{\psi'}_1(e_1,e_2,e_2)=e_1,\,\tilde{\psi'}_1(e_2,e_1,e_2)=-e_1,\, \tilde{\psi'}_1(e_2,e_2,e_2)=-e_1+e_2.$ 
Thus, we have $H^3=\left\langle\tilde{\psi'}_1\right\rangle$.
\item For $\tilde{A}'^2_{2}$, the $Z^3$ is $2$-dimensional generated by the $3$-cocycle generators : \\
$\begin{array}{ll}  
\tilde{\psi}'_1(e_1,e_1,e_2)=e_1-e_2,\\
\tilde{\psi}'_1(e_1,e_2,e_2)=e_1-e_2,\\
\tilde{\psi}'_1(e_2,e_1,e_1)=-e_1-e_2,\\
\end{array}$
$\begin{array}{ll}  
\tilde{\psi}'_1(e_2,e_2,e_1)=-e_1+e_2,\\
\tilde{\psi}'_2(e_2,e_1,e_2)=e_1+e_2,\\
\tilde{\psi}'_2(e_2,e_2,e_2)=e_1-e_2.
\end{array}$
Thus, we have $H^3=\left\langle\tilde{\psi'}_1,\tilde{\psi'}_2\right\rangle$. 
\item For $\tilde{A}'^2_{4}$, the $Z^3$ is $1$-dimensional generated by the $3$-cocycle generators : \\
$\tilde{\psi}'_1(e_1,e_2,e_2)=e_1,\quad \tilde{\psi}'_1(e_2,e_1,e_2)=-e_2$.
Thus, we have $H^3=\left\langle\tilde{\psi'}_1\right\rangle$.
\end{enumerate}
\subsection{Cohomology and Obstructions in $\mathcal{UA}ss_3$}
\subsubsection*{Computation of cohomology in $\mathcal{UA}ss_3$}
\begin{enumerate}
	\item For $\tilde{A}'^3_6$, the $Z^2$ is $13$-dimensional. Thus, we have $H^2=\left\langle\tilde{\varphi'_1},\dots,
	\tilde{\varphi'_{13}}\right\rangle$.
\item For $\tilde{A}'^3_7$, the $Z^2$ is $12$-dimensional. Thus, we have $H^2=\left\langle\tilde{\varphi'_1},\dots,
\tilde{\varphi'_{12}}\right\rangle$.
\item For $\tilde{A}'^3_8$, the $Z^2$ is $12$-dimensional. Thus, we have $H^2=\left\langle\tilde{\varphi'_1},\dots,
\tilde{\varphi'_{12}}\right\rangle$.
\item For $\tilde{A}'^3_9$, the $Z^2$ is $11$-dimensional. Thus, we have
 $H^2=\left\langle\tilde{\varphi'_1},\tilde{\varphi'_3},\tilde{\varphi'_4},\tilde{\varphi'_5},\tilde{\varphi'_6},\tilde{\varphi'_7},
\tilde{\varphi'_8}\right\rangle$.
\item For $\tilde{A}'^3_{10}$, the $Z^2$ is $12$-dimensional. Thus, we have $H^2=\left\langle\tilde{\varphi_4},\dots,\tilde{\varphi_{12}}\right\rangle$.
\item For $\tilde{A}'^3_{11}$, the $Z^2$ is $10$-dimensional. Thus, we have $H^2=\left\langle\tilde{\varphi_4},\dots,\tilde{\varphi_{10}}\right\rangle$.
\item For $\tilde{A}'^3_{12}$, the $Z^2$ is $10$-dimensional generated. 
Thus, we have 
$H^2=\left\langle\tilde{\varphi_1},\dots,\tilde{\varphi_{10}}\right\rangle\backslash\left\langle\tilde{ \varphi}_3,\tilde{\varphi}_5\right\rangle$.
\item For $\tilde{A}'^3_{14}$, the $Z^2$ is $11$-dimensional. Thus, we have 
$H^2=\left\langle\tilde{\varphi_3},\tilde{\varphi_4},\tilde{\varphi_5},\tilde{\varphi_7},\tilde{\varphi_{8}}\right\rangle$.
\end{enumerate}
\subsubsection*{Obstructions spaces in $\mathcal{UA}ss_3$}
\begin{enumerate}
	\item For $\tilde{A}'^3_{6}$, the $Z^3$ is $18$-dimensional.
Thus, we have $H^3=\left\langle\tilde{\psi}'_1,\dots,\tilde{\psi}'_{18}\right\rangle$.
\item For $\tilde{A}'^3_{7}$, the $Z^3$ is $18$-dimensional.
Thus, we have $H^3=\left\langle\tilde{\psi}'_1,\cdots,\tilde{\psi}'_{18}\right\rangle.$
\item For $\tilde{A}'^3_{8}$, the $Z^3$ is $2$-dimensional generated by the $3$-cocycle generators : \\
$\begin{array}{ll}  
\tilde{\psi}'_1(e_1,e_1,e_2)=e_2,\\
\tilde{\psi}'_1(e_1,e_1,e_3)=e_3,
\end{array}$
$\begin{array}{ll}  
\tilde{\psi}'_1(e_2,e_1,e_1)=-e_2,\\
\tilde{\psi}'_1(e_3,e_1,e_1)=-e_3,
\end{array}$
$\begin{array}{ll}  
\tilde{\psi}'_2(e_2,e_2,e_2)=e_2,\\
\tilde{\psi}'_2(e_2,e_2,e_3)=e_3.
\end{array}$.
$\begin{array}{ll}  
\tilde{\psi}'_2(e_2,e_3,e_2)=-e_3,\\
\tilde{\psi}'_2(e_2,e_3,e_3)=-e_2.
\end{array}$.

Thus, we have $H^3=\left\langle\tilde{\psi'}_1,\tilde{\psi}'_{2}\right\rangle.$
\item For $\tilde{A}'^3_{9}$, the $Z^3$ is $16$-dimensional.
Thus, we have 
$H^3=\left\langle\tilde{\psi}'_1,\dots,\tilde{\psi}'_{16}\right\rangle$. 
\item For $\tilde{A}'^3_{10}$, the $Z^3$ is $19$-dimensional.
Thus, we have $H^3=\left\langle\tilde{\psi}'_1,\dots,\tilde{\psi}'_{19}\right\rangle$.
\item For $\tilde{A}'^3_{11}$, the $Z^3$ is $20$-dimensional.
Thus, we have $H^3=\left\langle\tilde{\psi}'_1,\dots,\tilde{\psi}'_{20}\right\rangle$.
\item For $\tilde{A}'^3_{12}$, the $Z^3$ is $20$-dimensional.
Thus, we have $H^3=\left\langle\tilde{\psi'}_1,\dots,\tilde{\psi'}_{20}\right\rangle$. 
\item For $\tilde{A}'^3_{14}$, the $Z^3$ is $20$-dimensional.
Thus, we have $H^3=\left\langle\tilde{\psi'}_1,\dots,\tilde{\psi'}_{18}\right\rangle$. 
\end{enumerate}

\section{Deformations and irreducible components of Hom-associative algebras}
In this section, we aim to discuss the geometric classification of $\mathcal{HA}ss_n$ and  $\mathcal{UHA}ss_n$ for $n=2,3$. We use to this end one parameter formal deformation theory introduced first by Gerstenhaber for  associative algebras and extended to Hom-associative algebras in \cite{AEM,MS-Forum}.
\begin{definition}
Let $(A,\mu,\alpha)$    be a  Hom-associative algebra.
A formal  deformation of the  Hom-associative algebra $\A$ is given by  a $\K[[t]]$-bilinear map  $\mu_t : A[[t]]\times A[[t]]\longrightarrow A[[t]]$ of the form
$\mu_t=\sum_{i\geq0}t^i \mu_i$ where each $\mu_i$ is a $\K$-bilinear-map  $\mu_i : A\times A\rightarrow A$ (extended to be $\K[[t]]$-bilinear)  and $\mu_0=\mu$
 such that hold for $x, y, z\in A$ the following condition
 \begin{equation}\label{def ass}
 \mu_t(\mu_t(x,y),\alpha(z))=\mu_t(\alpha(x),\mu_t(y,z))
 \end{equation}


Suppose that $(A\left[\left[t\right]\right], \mu_{1,t}, \alpha_{1,t})$ and $(A\left[\left[t\right]\right], \mu'_{1,t}, \alpha'_{1,t})$ are
Hom-associative deformations of the  Hom-associative algebras $(A, \mu, \alpha)$. They are said  equivalent if there exists
a formal isomorphism between them, i.e. a $\K \left[\left[t\right]\right]$-linear map $\varphi_t$, compatible with both the deformed multiplications and the deformed  twisting maps, of the form 
$
\varphi_t=\displaystyle\sum_{i\geq 0}t^i\varphi_i,
$  
where the $\varphi_i$ are linear maps $\varphi_i:A\rightarrow A$ and $\varphi_0=id_A$. Compatibility with the deformed multiplications means that 
$\varphi_t\circ \mu_t=\mu'\circ(\varphi_t\otimes\varphi_t)$, compatibility to the twisting maps means 
$\varphi_t\circ \alpha_t=\alpha'\circ\varphi_t.$ 
\end{definition}
\begin{proposition}
Let $\mu_{1,t}=\phi^{-1}\circ \mu_2\circ (\phi\otimes \phi)$ and $\alpha_{1,t}=\phi^{-1}\circ\alpha_2 \circ\phi$. Then  if $(A, \mu_2, \alpha_2)$ is  Hom-associative then $(A\left[\left[t\right]\right], \mu_{1,t}, \alpha_{1,t})$ is Hom-associative.
\end{proposition}
\begin{proof}
By straightforward computation, we have
\begin{align*}
\mu_{1,t}(\alpha_{1,t}(x), \mu_{1,t}(y, z))
&=\phi^{-1}\mu_2(\phi(\phi^{-1}\circ\alpha_2\circ\phi(x)),\phi\phi^{-1}\circ\mu_2(\phi(y),\phi(z)))\\
&=\phi^{-1}\mu_2(\alpha_2\circ\phi(x), \mu_2(\phi(y),\phi(z)))\\
&=\phi^{-1}\mu_2(\mu_2(\phi(x), \phi(y)),\alpha_2\circ\phi(z))\\
&=\phi^{-1}\mu_2(\phi\circ\phi^{-1}(\mu_2(\phi(x),\phi(y))),\phi\circ\phi^{-1}\alpha_2\phi(z)))\\
&=\mu_{1,t}(\mu_{1,t}(x,y), \alpha_{1,t}(z)).
\end{align*}
\end{proof}

\begin{definition}
A Hom-associative algebra $A$ is called formally rigid, if every formal  deformation of $A$ is trivial. It is called 
geometrically rigid, if its orbid $\vartheta(\mu)$ is open in $\mathcal{HA}ss_n$. Then $\overline{\vartheta(\mu)}$ is an irreducible 
component of $\mathcal{HA}ss_n$.  
\end{definition}
\begin{remark}
Any irreducible component $\mathcal{C}$ of $\mathcal{HA}ss_n$ containing $A$ also contains all degenerations of $A$. Indeed, we have 
$\vartheta(\mu)\subset \mathcal{C}$ so that $\overline{\vartheta(\mu)}$ is contained in $\mathcal{C}$, since  $\mathcal{C}$ is closed.       
\end{remark}

\begin{proposition}
The irreducible components of $\mathcal{HA}ss_2$ are the Zariski closure of orbits of Hom-associative algebras 
$\Omega=\left\{A^2_3, A^2_5\right\}$.
\end{proposition}
\begin{center}
\includegraphics[scale=0.4]{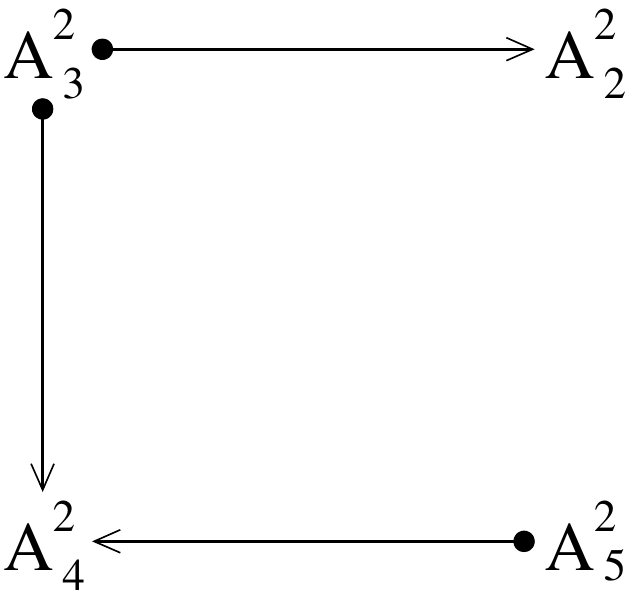} \\[1ex]
	Irreducible Components $\mathcal{HA}ss_2$
\end{center}

\begin{proposition}
The irreducible components of $\mathcal{HA}ss_3$ are the Zariski closure of orbits of Hom-associative algebras 
$\Omega=\left\{A^3_2, A^3_5, A^3_9\right\}$.
\end{proposition}
\begin{center}
	\includegraphics[scale=0.4]{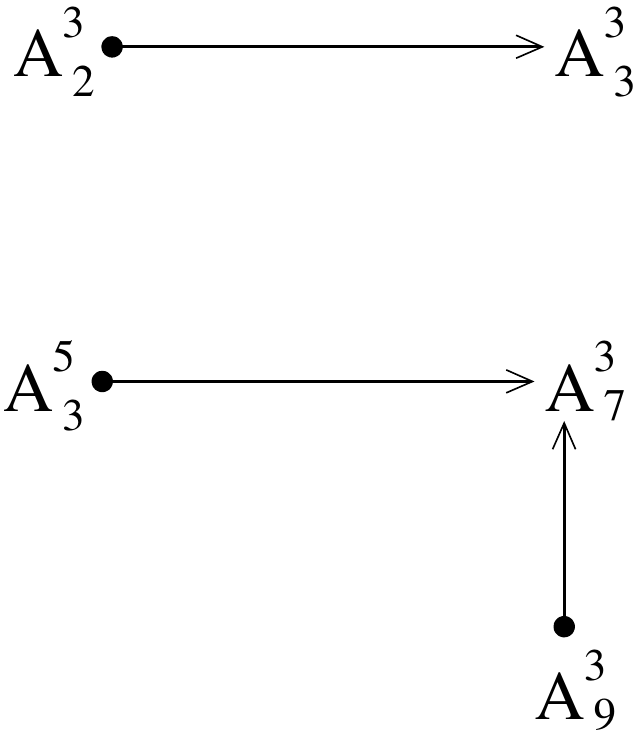} \\[1ex]
Irreducible components of $\mathcal{A}ss_3$
\end{center}

\begin{proposition}
The irreducible components of $\mathcal{UHA}ss_2$ are the Zariski closure of orbits of Hom-associative algebras 
$\Omega=\left\{A'^2_3, A'^2_4\right\}$.
\end{proposition}
\begin{center}
\includegraphics[scale=0.4]{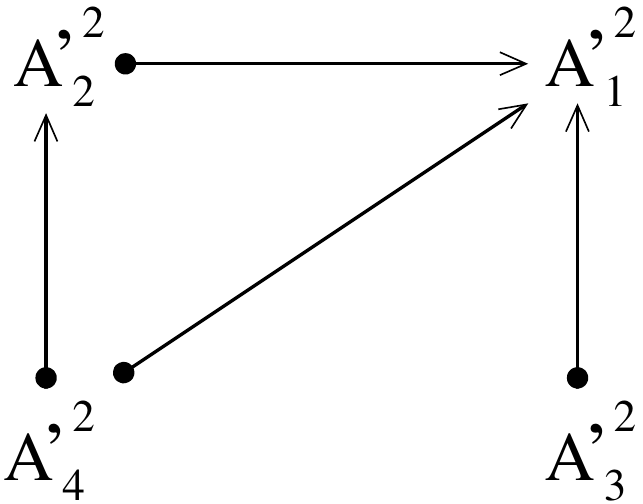} \\[1ex]
	Irreducible components of $\mathcal{UHA}ss_2$
\end{center}
%

\begin{proposition}
The irreducible components of $\mathcal{UHA}ss_3$ are the Zariski closure of orbits of Hom-associative algebras 
$\Omega=\left\{A'^3_9, A'^3_{10}, A'^3_{11},A'^3_{12},A'^3_{13},A'^3_{14},A'^3_{15}\right\}$.
\end{proposition}

\begin{center}
	\includegraphics[scale=0.6]{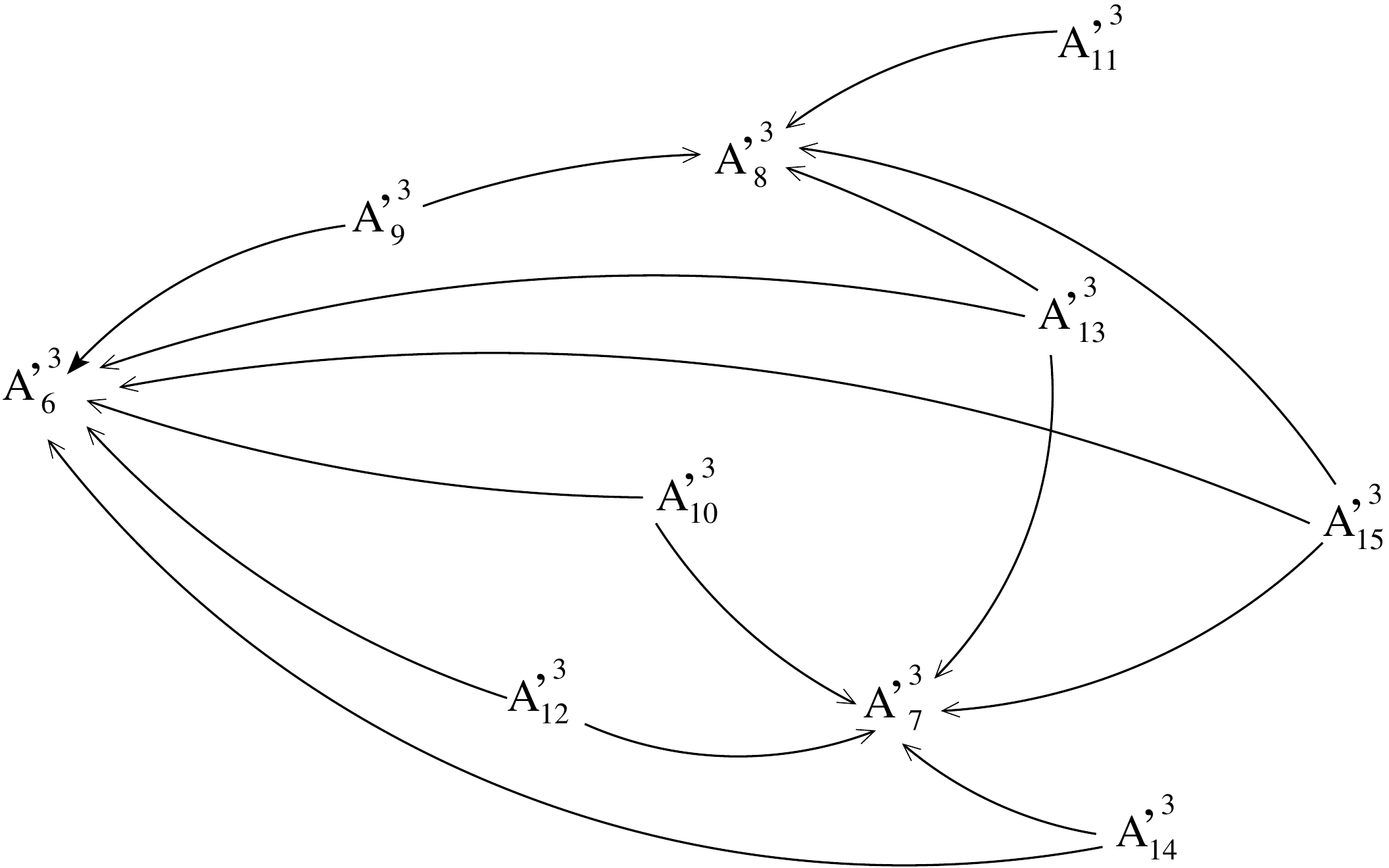} \\[1em]
  Irreducible components of $\mathcal{UHA}ss_3$
\end{center}

\begin{proposition}
The irreducible components of $\mathcal{UA}ss_3$ are the Zariski closure of orbits of Hom-associative algebras 
$\Omega=\left\{\tilde{A}^3_{10}, \tilde{A}^3_{14}\right\}$.
\end{proposition}

\begin{center}
		\includegraphics[scale=0.6]{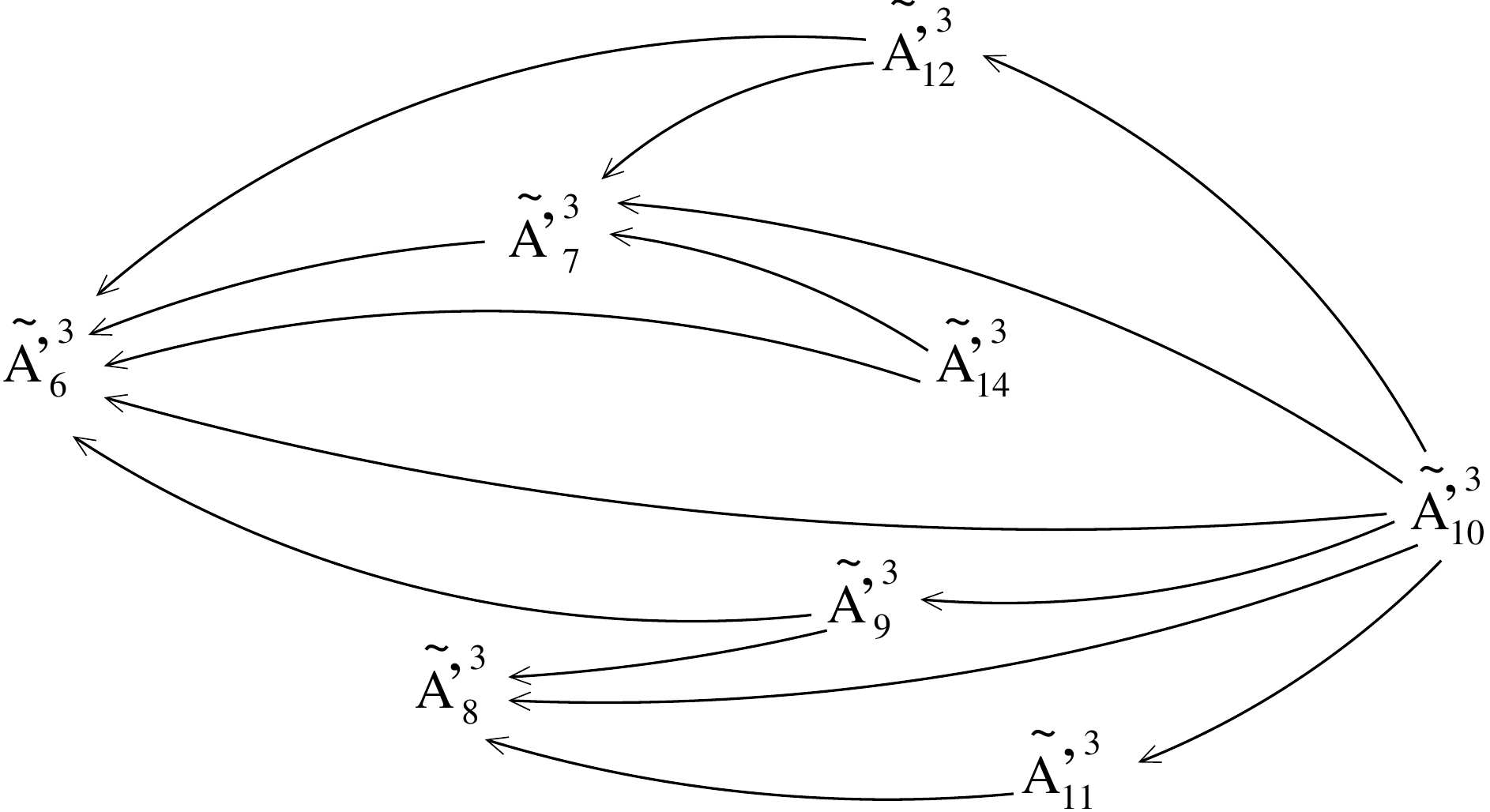} \\[1em]
 Irreducible components of $\mathcal{A}ss_3$
\end{center}

\end{document}